\documentclass{amsart}

\usepackage{amssymb,amscd,amsmath,hyperref,color,enumerate,tikz-cd}
\usepackage[all,cmtip]{xy}

\newcommand{\N}{{\ensuremath{\mathbb{N}}}}
\newcommand{\Z}{{\ensuremath{\mathbb{Z}}}}

\newcommand{\C}{{\ensuremath{\mathbb{C}}}}
\def\BB{\mathrm{B}}
\def\KK{\mathrm{K}}
\def\MM{\mathrm{M}}

\def\H{\mathcal{H}}

\def\Z{\mathcal{Z}}
\def\S{\mathcal{S}}

\def\lm{\lambda}
\def\a{\alpha}
\def\ep{\varepsilon}

\def\CE{\mathcal{E}}
\def\P{\mathrm{Prim}}
\def\M{\mathrm{Max}}
\def\I{\mathrm{Id}}
\def\G{\mathrm{Glimm}}

\def\s{\mathop{\sigma}}

\def\di{\mathop{\mathrm{diag}}}
\def\int{\mathop{\mathrm{Int}}}

\def\cq{\mathrm{CQ}}

\newtheorem{proposition}{Proposition}[section]
\newtheorem{lemma}[proposition]{Lemma}

\newtheorem{theorem}[proposition]{Theorem}
\newtheorem{corollary}[proposition]{Corollary}
\theoremstyle{definition}
\newtheorem{remark}[proposition]{Remark}

\newtheorem{definition}[proposition]{Definition}
\newtheorem{example}[proposition]{Example}

\numberwithin{equation}{section}

\begin{document}

\title[]{The centre-quotient property and weak centrality for $C^*$-algebras}

\author{Robert J. Archbold}
\author{Ilja Gogi\'c}

\address{R.~J.~Archbold,  Institute of Mathematics, University of Aberdeen, King's College, Aberdeen AB24 3UE, Scotland, United Kingdom}
\email{r.archbold@abdn.ac.uk}

\address{I.~Gogi\'c, Department of Mathematics, Faculty of Science, University of Zagreb, Bijeni\v{c}ka 30, 10000 Zagreb, Croatia}
\email{ilja@math.hr}

\keywords{$C^*$-algebra, centre-quotient property, weak centrality, commutator}

\subjclass[2010]{Primary 46L05; Secondary 46L06}

%\date{\today}

\begin{abstract}
We give a number of equivalent conditions (including weak centrality) for a general $C^*$-algebra to have the centre-quotient property. We show that every $C^*$-algebra $A$ has a largest weakly central ideal $J_{wc}(A)$. For an ideal $I$ of a unital $C^*$-algebra $A$, we find a necessary and sufficient condition for a central element of $A/I$ to lift to a central element of $A$. This leads to a characterisation of the set $V_A$ of elements of an arbitrary $C^*$-algebra $A$ which prevent $A$ from having the centre-quotient property. The complement $\mathrm{CQ}(A):= A \setminus V_A$ always contains $Z(A)+J_{wc}(A)$ (where $Z(A)$ is the centre of $A$), with equality if and only if $A/J_{wc}(A)$ is abelian. Otherwise, $\mathrm{CQ}(A)$ fails spectacularly to be a $C^*$-subalgebra of $A$.
\end{abstract}

\maketitle

\section{Introduction}

Let $A$ be a $C^*$-algebra with centre $Z(A)$. If $I$ is a closed two-sided ideal of $A$, it is immediate that
\begin{equation}\label{cent}
(Z(A)+I)/I= q_I(Z(A)) \subseteq Z(A/I),
\end{equation}
where $q_I : A \to A/I$ is the canonical map. A $C^*$-algebra $A$ is said to have the \emph{centre-quotient property} (\cite{Vest}, \cite[Section~2.2]{ArcTh} and \cite[p.~2671]{ART}) if for any closed two-sided ideal $I$ of $A$, equality holds in (\ref{cent}). For the sake of brevity we shall usually refer to the centre-quotient property as the \emph{CQ-property}.

\smallskip

In 1971, Vesterstr\o m \cite{Vest} proved the following theorem.

\begin{theorem}[Vesterstr\o m]\label{VT}
If $A$ is a unital $C^*$-algebra, then the following conditions are equivalent:
\begin{itemize}
\item[(i)] $A$ has the CQ-property.
\item[(ii)] $A$ is weakly central, that is for any pair of maximal ideals  $M$ and $N$  of $A$,
$M \cap Z(A) =N \cap Z(A)$ implies $M=N$.
\end{itemize}
\end{theorem}
Weakly central $C^*$-algebras were introduced by Misonou and Nakamura in \cite{MN,Mis} in the unital context. The most prominent examples of weakly central $C^*$-algebras $A$ are those satisfying the Dixmier property, that is for each $x \in A$ the closure of the convex hull of the unitary orbit of $x$ intersects $Z(A)$ \cite[p.~275]{Arc4}. In particular, von Neumann algebras are weakly central (see \cite[Th\'{e}or\`{e}me~7]{Dix} and \cite[Theorem~3]{Mis}). It was shown by Haagerup and Zsid\'o in \cite{HZ} that a unital simple $C^*$-algebra satisfies the Dixmier property if and only if it admits at most one tracial state. In particular, weak centrality does not imply the Dixmier property. However, in \cite{Mag} Magajna gave a characterisation of weak centrality in terms of averaging involving unital completely positive elementary operators. Recently, Robert, Tikuisis and the first-named author found the exact gap between
weak centrality and the Dixmier property for unital $C^*$-algebras \cite[Theorem~2.6]{ART} and showed that a postliminal $C^*$-algebra has the (singleton) Dixmier property if and only if it has the CQ-property \cite[Theorem~2.12]{ART}. Also, in a recent paper \cite{BG}, Bre\v sar and the second-named author studied an analogue of the CQ-property in a wider algebraic setting (so called `centrally stable algebras').

\smallskip

In this paper we study weak centrality, the CQ-property and several equivalent conditions for general $C^*$-algebras that are not necessarily unital. We then investigate the failure of weak centrality in two different ways. Firstly, we show that every $C^*$-algebra $A$ has a largest weakly central ideal $J_{wc}(A)$, which can be readily determined in several examples. Secondly, we study the set $V_A$ of individual elements of $A$ which prevent the weak centrality (or the CQ-property) of $A$. The set $V_A$ is contained in the complement of $J_{wc}(A)$ and, in certain cases, is somewhat smaller than one might expect. In the course of this, we address a fundamental lifting problem that is closely linked to the CQ-property: for a fixed ideal $I$ of a unital $C^*$-algebra $A$, we find a necessary and sufficient condition for a central element of $A/I$ to lift to a central element of $A$.

\smallskip

The paper is organised as follows. After some preliminaries in Section 2, the main results are obtained in Section \ref{S3} and Section \ref{S4}. In Section \ref{S3}, we study weak centrality and the CQ-property for arbitrary $C^*$-algebras. In the non-unital context, the appropriate maximal ideals for the definition of weak centrality are the modular maximal ideals (Definition \ref{defwc}). In Theorem \ref{chQC}, we give a number of conditions (including the CQ-property) that are equivalent to the weak centrality of a $C^*$-algebra $A$. In Theorem \ref{thmJcq}, we show that every $C^*$-algebra $A$ has a largest ideal $J_{wc}(A)$ that is weakly central. In doing so, we obtain a formula for $J_{wc}(A)$ in terms of the set $T_A$ of those modular maximal ideals of $A$ which witness the failure of the weak centrality of $A$. This formula leads easily to the explicit description of $J_{wc}(A)$  in a number of examples. For example, $J_{wc}(A) = \{0\}$ when either $A$ is the rotation algebra (the $C^*$-algebra of the discrete three-dimensional Heisenberg group, Example \ref{ex:Heis}) or $A=C^*(\mathbb{F}_2)$ (the full $C^*$-algebra of the free group on two generators, Example \ref{ex:freegroup}), and $J_{wc}(A) = \KK(\H)$ for Dixmier's classic example of a $C^*$-algebra in which the Dixmier property fails (Example \ref{ex:Dix}). We also obtain the stability of weak centrality and the CQ-property in the context of arbitrary $C^*$-tensor products (Theorem \ref{thm:tenprod}).

\smallskip

In Section \ref{S4}, we undertake the more difficult task of describing the individual elements which prevent a $C^*$-algebra $A$ from having the CQ-property. We say that an element $a\in A$ is a \emph{CQ-element} if for every closed two-sided ideal $I$ of $A$, $a+I \in Z(A/I)$ implies $a \in Z(A)+I$
(Definition \ref{def:cqel}). We denote by $\cq(A)$ the set of all CQ-elements of $A$.

Clearly, $A$ has the CQ-property if and only if $\cq(A)=A$ and the complement $V_A:= A\setminus \cq(A)$ is precisely the set of elements which prevent the CQ-property for $A$. For an ideal $I$ of a unital $C^*$-algebra $A$, we use the complete regularization map, the Tietze extension theorem and the Dauns-Hofmann theorem to obtain a necessary and sufficient condition for a central element of $A/I$ to lift to a central element of $A$ (Theorem \ref{thmlift}). This then leads to a description of $V_A$ (and hence $\cq(A)$) for an arbitrary $C^*$-algebra $A$ in terms of the subset $T_A$ (Theorem \ref{thmcqel}).  

We show that $\cq(A)$ contains $Z(A)+J_{wc}(A)$ (Corollary \ref{corsumsub}), all commutators $[a,b]$ ($a,b\in A$) and all products $ab$, $ba$ where $a\in A$ and $b$ is a quasi-nilpotent element of $A$ (Proposition \ref{prop:commutators}). It follows from this, together with a result of Pop \cite[Theorem~1]{Pop}, that if $A$ is not weakly central (so that $\cq(A)\neq A$), $\cq(A)$ contains the norm-closure $\overline{[A,A]}$ of $[A,A]$ (the linear span of all commutators in $A$) if and only if all quotients $A/M$ ($M \in T_A$) admit tracial states (Theorem \ref{thm:comtstates}). In particular, if $A$ is postliminal or an AF-algebra, then $\overline{[A,A]}\subseteq \cq(A)$ (Corollary \ref{cor:postAF}). On the other hand, if the tracial condition is not satisfied, then $\cq(A)$ does not even contain $[A,A]$ (Theorem \ref{thm:comtstates} (b)). 

Further, we show that for any $C^*$-algebra $A$ the following conditions are equivalent:
\begin{itemize}
\item[(i)] $A/J_{wc}(A)$ is abelian.
\item[(ii)] $\cq(A)=Z(A)+J_{wc}(A)$.
\item[(iii)] $\cq(A)$ is closed under addition.
\item[(iv)] $\cq(A)$ is closed under multiplication. 
\item[(v)] $\cq(A)$ is norm-closed.
\end{itemize}
(Theorem \ref{thm:cqequiv}). If $A$ is postliminal or an AF-algebra, then the conditions (i)-(v) are also equivalent to the condition 
\begin{itemize}
\item[(vi)] For any $x \in \cq(A)$, $x^n \in \cq(A)$ for all positive integers $n$.
\end{itemize}
(Corollary \ref{cor:postlim}). We also show that (vi) does not have to imply (i)-(v) for general (separable nuclear) $C^*$-algebras (Example \ref{ex:JiSutp}). The methods for these results involve the lifting of nilpotent elements, commutators and simple projectionless $C^*$-algebras.

We finish with an example of a separable continuous-trace $C^*$-algebra $A$ for which $J_{wc}(A)=Z(A)=\{0\}$ but $\cq(A)$ is norm-dense in $A$ (Example \ref{exRobTh}). In other words, although no non-zero ideal of $A$ has the CQ-property, the set $V_A$ of elements which prevent the CQ-property of $A$ has empty interior in $A$.

\section{Preliminaries}\label{prel}

Throughout this paper $A$ will be a $C^*$-algebras with the centre $Z(A)$.  By $\S(A)$ we denote the set of all states on $A$. As usual, if $x,y \in A$ then $[x,y]$ stands for the commutator $xy-yx$. If $A$ is non-unital, we denote the (minimal) unitization of $A$ by $A^\sharp$. If $A$ is unital we assume that $A^\sharp=A$.

By an ideal of $A$ we shall always mean a closed two-sided ideal. If $X$ is a subset of $A$, then $\I_A(X)$ denotes the  ideal of $A$ generated by $X$. An ideal $I$ of $A$ is said to be \emph{modular} if the quotient $A/I$ is unital. If 
$I$ is an ideal of $A$ then it is well-known that $Z(I)=I \cap Z(A)$.

The set of all primitive ideals of $A$ is denoted by $\P(A)$. As usual, we equip $\P(A)$ with the Jacobson topology. It is well-known that any proper modular ideal of $A$ (if such exists) is contained in some modular maximal ideal of $A$ (see e.g. \cite[Lemma~1.4.2]{Kan}) and that all modular maximal ideals of $A$ are primitive. We denote the set of all modular maximal ideals of $A$ by $\M(A)$, so that $\M(A) \subseteq \P(A)$. Note that $\M(A)$ can be empty (e.g. the algebra $A=\KK(\H)$ of compact operators on a separable infinite-dimensional Hilbert space $\H$). 

\begin{remark}\label{nmod}
Every non-modular primitive ideal of a $C^*$-algebra $A$ contains $Z(A)$. Indeed, if $z \in Z(A)$ then for $P \in \P(A)$, $z+P\in Z(A/P)$ and so is either zero or a multiple of the identity in $A/P$ if $A/P$ has one. Therefore, if $P$ is non-modular, we have $z \in P$ for all $z \in Z(A)$, so $Z(A)\subseteq P$.  In particular, if the set of all non-modular primitive ideals of $A$ is dense in $\P(A)$, then $Z(A)=\{0\}$.
\end{remark}

For any subset $S \subseteq \P(A)$ we define its kernel $\ker S$ as the intersection of all elements of $S$. For the case $S=\emptyset$, we define $\ker S=A$. Note that $S$ is closed in $\P(A)$ if and only if for any $P \in \P(A)$, $\ker S \subseteq P$ implies $P \in S$.

For any ideal $I$ of $A$ we define the following two subsets of $\P(A)$:
$$\P_I(A):=\{P \in \P(A) : \ I \not \subseteq P\} \quad \mbox{and} \quad \P^I(A):=\{P \in \P(A) : \ I \subseteq P\}.$$
Then $\P_I(A)$ is an open subset of $\P(A)$ homeomorphic to $\P(I)$ via the map $P \mapsto P\cap I$, while $\P^I(A)$ is a closed subset of $\P(A)$ homeomorphic to $\P(A/I)$ via the map $P \mapsto P/I$ (see e.g. \cite[Proposition~A.27]{RW}).
Similarly, we introduce the following subsets of $\M(A)$:
$$\M_I(A):=\{M  \in \M(A) : \ I \not \subseteq M\} \quad \mbox{and} \quad \M^I(A):=\{M \in \M(A) : \ I \subseteq M\}.$$

We shall frequently use the next simple fact which is probably well-known but as we have been unable to find a reference we include a proof for completeness. 

\begin{lemma}\label{lemmaxJ} Let $A$ be a $C^*$-algebra and let $I$ be an arbitrary ideal of $A$. 
Then the assignment  $M \mapsto M \cap I$ defines  a homeomorphism from the set $\M_I(A)$ onto the set $\M(I)$.

\end{lemma}
\begin{proof}
For $M \in \M_I(A)$ set $\psi(M):=M \cap I$. 

Let $M \in \M_I(A)$. Since $I\nsubseteq M$, by maximality and modularity of $M$ we have $I+M=A$ and $A/M$ is a simple unital $C^*$-algebra. 
Using the canonical isomorphism $I/(M\cap I)\cong (I+M)/M=A/M$, 
we conclude that $I/(M\cap I)$ is also a simple unital $C^*$-algebra, so $\psi(M)=M \cap I \in \M(I)$. 

The injectivity of the map $\psi$ follows directly from the injectivity of the assignment $\P_I(A)\to \P(I)$, $P \mapsto P \cap I$, and the fact that all modular maximal ideals of $A$ are primitive.  

To show the surjectivity of $\psi$, choose an arbitrary $N \in \M(I)$. Then there is $M \in \P_I(A)$ such that
$N=M\cap I$. Since $I/N \cong (I+M)/M$, it follows that $(I+M)/M$ is a unital ideal of the primitive $C^*$-algebra $A/M$. This forces $(I+M)/M=A/M$, so $A/M$ is a unital simple $C^*$-algebra. Thus, $M \in \M_I(A)$.

Finally, since $\psi$ is a restriction of a canonical homeomorphism $\P_I(A)\to \P(I)$ on $\M_I(A)$ with the image $\M(I)$, it is itself a homeomorphism.
\end{proof}

\begin{remark}
It is also easy to see that for any ideal $I$ of a $C^*$-algebra $A$, the assignment $M \mapsto M/I$ defines a homeomorphism from the set $\M^I(A)$ onto the set $\M(A/I)$, but we shall not use this fact in this paper. 
\end{remark}

If $A$ is a $C^*$-algebra and $P$ a primitive ideal of $A$ such that $Z(A) \nsubseteq P$, then $P \cap Z(A)$ is a maximal ideal of $Z(A)$. In particular, if $A$ is unital, then the map 
$$\P(A) \to \M(Z(A)) \qquad \mbox{defined by} \qquad P \mapsto  P \cap Z(A)$$
if a well-defined continuous surjection. We shall continue to assume that $A$ is unital in this paragraph and the next. For all $P, Q \in \mathrm{Prim}(A)$ we define 
$$
P \approx Q \qquad  \mbox{if} \qquad P \cap Z(A)= Q \cap Z(A). 
$$
By the Dauns-Hofmann theorem \cite[Theorem~A.34]{RW}, there exists an isomorphism
$$\Psi_A : Z(A) \to C(\P(A)) \qquad \mbox{such that} \qquad z+P=\Psi_A(z)(P)1+P$$
for all $z \in Z(A)$ and $P \in \P(A)$ (note that $\P(A)$ is compact, as $A$ is unital \cite[II.6.5.7]{Bla}). Hence, for all $P,Q \in \P(A)$ we have 
$$P \approx Q \qquad \Longleftrightarrow \qquad f(P)=f(Q) \quad \mbox{for all } f \in C(\P(A)).$$  
Note that $\approx$ is an equivalence relation on $\P(A)$ and the equivalence
classes are closed subsets of $\P(A)$. It follows there is one-to-one correspondence between the quotient set $\mathrm{Prim}(A)/\approx$
and a set of ideals of $A$ given by
$$[P]_\approx \longleftrightarrow \ker [P]_\approx \quad (P \in \mathrm{Prim}(A)),$$
where $[P]_\approx$ denotes the equivalence class of $P$. The set of ideals obtained in this way is denoted by $\G(A)$, and its
elements are called \textit{Glimm ideals} of $A$. The quotient map 
$$\phi_A : \mathrm{Prim}(A) \to
\mathrm{Glimm}(A), \qquad \phi_A(P):=\ker [P]_\approx$$ is known as the \textit{complete regularization map}. 
We equip $\G(A)$ with the quotient topology, which coincides with the complete regularization topology, since $A$ is unital. In this way $\G(A)$ becomes a compact Hausdorff space. In fact, $\G(A)$ is homeomorphic to $\M(Z(A))$ via the assignment 
$G \mapsto G \cap Z(A)$ (see \cite{AS} for further details).

\begin{definition}\label{qcdef}
A $C^*$-algebra $A$ is said to be \textit{quasi-central} if no primitive ideal of $A$ contains $Z(A)$.
\end{definition}

Quasi-central $C^*$-algebras were introduced by Delaroche in \cite{Del2}. 
We have the following useful characterisation of quasi-central $C^*$-algebras.

\begin{proposition}\cite[Proposition~1]{Arc}\label{qcchar} Let $A$ be a $C^*$-algebra. The following
conditions are equivalent:
\begin{itemize}
\item[(i)] $A$ is quasi-central.
\item[(ii)] $A$ admits a central approximate unit, i.e. there exists an approximate unit $(e_\alpha)$ of $A$ such that $e_\alpha \in Z(A)$ for all $\alpha$.
\end{itemize}
\end{proposition}

\begin{remark}
It is easily seen from Proposition \ref{qcchar} (ii) that quasi-centrality passes to quotients and tensor products. 
\end{remark}

We have the following prominent examples of quasi-central $C^*$-algebras.
\begin{example}\label{exqc}
\begin{itemize}
\item[(a)] Every unital $C^*$-algebra is obviously quasi-central.
\item[(b)] Every abelian $C^*$-algebra is quasi-central. More generally, a $C^*$-algebra $A$ is said to be \textit{$n$-homogeneous} if all irreducible representations of $A$ have the same finite dimension $n$ (note that the abelian $C^*$-algebras are precisely $1$-homogeneous $C^*$-algebras). Then by \cite[Theorem~4.2]{Kap2} $\mathrm{Prim}(A)$ is a (locally compact) Hausdorff space and by a well-known theorem of Fell \cite[Theorem~3.2]{Fell} and Tomiyama-Takesaki \cite[Theorem~5]{TT} there is a locally trivial bundle $\CE $ over $\mathrm{Prim}(A)$ with fibre $\MM_n(\C)$ and structure group $\mathrm{Aut}(\MM_n(\C))\cong PU(n)$ (the projective unitary group) such that $A$ is isomorphic to the $C^*$-algebra $\Gamma_0(\CE )$ of continuous sections of $\CE $ that vanish at infinity. Using the local triviality of the underlying bundle $\CE $ one can now easily check that $A\cong \Gamma_0(\CE )$ is quasi-central (see also \cite[p.~236]{Kap2}).
\item[(c)] For a locally compact group $G$, the following conditions are equivalent:
\begin{itemize}
\item[(i)] The full group $C^*$-algebra $C^*(G)$ is quasi-central.
\item[(ii)] The reduced group $C^*$-algebra $C^*_{r}(G)$ is quasi-central.
\item[(iii)] $G$ is an SIN-group (that is, the identity has a base of neighbourhoods that are invariant under conjugation by elements of $G$).
\end{itemize}
(See \cite[Corollary~1.3]{Lo} and the remark which follows it.)
\end{itemize}
\end{example}

\begin{remark}\label{remqcid}
Let $A$ be an arbitrary $C^*$-algebra.
\begin{itemize}
\item[(a)] Using the Hewitt-Cohen factorization theorem (see e.g. \cite[Theorem~A.6.2]{BLM}) we have 
$$K_A:=\I_A(Z(A))=Z(A)A=\{za : \, z \in Z(A),\, a \in A\}$$
(finite sums are not needed). In particular, $A$ is quasi-central if and only if $A=Z(A)A$ (\cite[Proposition~3.2]{Gog}).
\item[(b)] The ideal $K_A$ is in fact the largest quasi-central ideal of $A$ \cite{Del2}. Indeed, $K_A$ contains $Z(A)$,  so $Z(K_A)= K_A\cap Z(A)=Z(A)$. Therefore, $Z(K_A)K_A=K_A$, so $K_A$ is quasi-central. On the other hand, if $K$ is an arbitrary quasi-central ideal of $A$, then (a) implies $K=Z(K)K$. Since $Z(K)=K \cap Z(A) \subseteq Z(A)$, $K \subseteq Z(A)A=K_A$.
\item[(c)] Since $P \in \P(A)$ contains $Z(A)$ if and only if $P$ contains $K_A$, it follows
$$K_A=\ker\{P \in \P(A) : \, Z(A) \subseteq P\}.$$ 
\item[(d)] If $A$ is quasi-central, then all primitive ideals of $A$ are modular. This follows directly from Remark \ref{nmod}.
\end{itemize}
\end{remark}

The following well-known example shows that the converse of Remark \ref{remqcid} (d) is not true in general. First recall that a $C^*$-algebra $A$ is called \emph{$n$-subhomogeneous} ($n \in \N$) if all irreducible representations of $A$ have dimension at most $n$ and $A$ also admits an $n$-dimensional irreducible representation.
\begin{example}\label{ex2subnqc}
Consider the $C^*$-algebra $A$ that consists of all continuous functions 
$f:[0,1]\to \MM_2(\C)$ such that $f(1)=\di(\lambda(f),0)$, for some scalar $\lambda(f)\in \C$. Since $A$ is $2$-subhomogeneous, all primitive ideals of $A$ are modular. On the other hand, $$Z(A)=\{\di(f,f) : \, f\in C([0,1]), \, f(1)=0\}$$ is contained in the kernel of the one-dimensional (hence irreducible) representation $\lambda : A \to \C$, defined by the assignment $\lambda : f \mapsto \lambda(f)$. Hence, $A$ is not quasi-central. In fact the largest quasi-central ideal $K_A$ of $A$ consists of all $f \in A$ such that $f(1)=0$, since the primitive ideals of this ideal have the form
$$\{f \in A: \, f(t)=0 \, \mbox{ and } \, f(1)=0\}$$
for $t \in [0,1)$.
\end{example}

\section{Characterisations of $C^*$-algebras with the CQ-property}\label{S3}

We begin this section with the following $C^*$-algebraic version of \cite[Proposition 2.1]{BG} in which for $a\in A$,
$$[a,A]:=\{[a,x] : \, x \in A\}.$$
Since the proof requires only obvious changes, we omit it.

\begin{proposition}\label{ldef}
Let $A$ be a $C^*$-algebra. The following conditions are equivalent:
\begin{enumerate}
\item[{\rm (i)}] $A$ has the CQ-property.
\item[{\rm (ii)}] For every $*$-epimorphism $\phi : A \to B$, where $B$ is another $C^*$-algebra,  $\phi(Z(A))=Z(B)$.
\item[{\rm (iii)}] For every $a\in A$, $a\in Z(A) +{\rm Id}_A([a,A])$.
\end{enumerate}
\end{proposition}

The next fact was obtained in \cite[Lemma~2.2.3]{ArcTh} but we include the details here for completeness.

\begin{proposition}\label{idquCQ}
If a $C^*$-algebra $A$ has the CQ-property, so do all ideals and quotients of $A$.
\end{proposition}

\begin{proof}
Assume that $A$ has the CQ-property and let $I$ be an ideal of $A$.

If $J$ is an ideal of $I$, then $J$ is an ideal of $A$ and $I/J$ is an ideal of $A/J$. The CQ-property of $A$ implies
\begin{equation}\label{eqidqu}
Z(I/J)=(I/J) \cap Z(A/J)=(I/J) \cap ((Z(A)+J)/J).
\end{equation}
Let $a \in I$ such that $a+J \in (Z(A)+J)/J$. Then there is $z \in Z(A)$ such that $a-z \in J\subseteq I$, so
$z \in I \cap Z(A)=Z(I)$. It follows that
$$(I/J) \cap ((Z(A)+J)/J)=(Z(I)+J)/J,$$
so by (\ref{eqidqu}), $$Z(I/J)=(Z(I)+J)/J.$$ Therefore, $I$ has the CQ-property.

We now show that $A/I$ has the CQ-property. Let $q_I : A \to A/I$ be the canonical map
and $\phi : A/I \to B$ any $*$-epimorphism, where $B$ is another $C^*$-algebra.
Then $\phi \circ q_I : A \to B$ is a $*$-epimorphism, so the CQ-property of $A$ implies
$$Z(B)=Z((\phi\circ q_I)(A))=(\phi \circ q_I)(Z(A))=\phi(q_I(Z(A)))=\phi(Z(A/I)).$$
Therefore, $A/I$ has the CQ-property.
\end{proof}

\begin{proposition}\label{propzerocen}
For a $C^*$-algebra $A$ the following conditions are equivalent:
\begin{itemize}
\item[{\rm (i)}]  $Z(A) = \{0\}$ and $A$ has the CQ-property.
\item[{\rm (ii)}] Every primitive ideal of $A$ is non-modular. 
\end{itemize}
\end{proposition}
\begin{proof}
(i) $\Longrightarrow$ (ii). Assume that  $Z(A) = \{0\}$ and that $A$ has the CQ-property. Then for any $P \in \P(A)$ we have $Z(A/P)=(Z(A)+P)/P=\{0\}$, so $P$ is non-modular.

\smallskip 

(ii) $\Longrightarrow$ (i). Assume that all primitive ideals of $A$ are non-modular. By Remark \ref{nmod} $Z(A)=\{0\}$. Also, for any ideal $I$ of $A$, all primitive ideals of $A/I$ are non-modular, so Remark \ref{nmod} again implies $Z(A/I)=\{0\}$. Thus, $A$ has the CQ-property.
\end{proof}

The following result was obtained in \cite[Proposition~2.2.4]{ArcTh} but we give a shorter argument in one direction by using the method of \cite[Proposition~2.15]{BG}.

\begin{proposition}\label{propcqunit}
For a non-unital $C^*$-algebra $A$ the following conditions are equivalent:
\begin{itemize}
\item[(i)] $A$ has the CQ-property.
\item[(ii)] $A^\sharp$ has the CQ-property.
\end{itemize}
\end{proposition}
\begin{proof}
(i) $\Longrightarrow$ (ii). Suppose that $A$ has the CQ-property and let $\lambda 1 + a \in A^\sharp$, where $a \in A$ and $\lambda \in \C$. Then, by Proposition \ref{ldef}, we have $a\in Z(A)+{\rm Id}_A([a,A])$. Since 
$${\rm Id}_{A^\sharp}([\lambda 1+ a,A^\sharp])={\rm Id}_{A}([a,A]),$$
it follows that $a \in Z(A)+{\rm Id}_{A^\sharp}([\lambda 1+ a,A^\sharp])$. Since $Z(A^\sharp)=\C1+Z(A)$ we conclude that
$$\lambda 1+ a \in  Z(A^\sharp)+{\rm Id}_{A^\sharp}([\lambda 1 + a,A^\sharp]).$$ 
Therefore, by Proposition \ref{ldef}, $A^\sharp$ has the CQ-property. 

(ii) $\Longrightarrow$ (i). Since $A$ is an ideal of $A^\sharp$, this follows directly from Proposition \ref{idquCQ}.
\end{proof}

We now extend the notion of weak centrality to arbitrary $C^*$-algebras.
\begin{definition}\label{defwc}
We say that a $C^*$-algebra $A$ is \emph{weakly central} if the following two conditions are satisfied:
\begin{itemize}
\item[(a)] No modular maximal ideal of $A$ contains $Z(A)$.
\item[(b)] For each pair of modular maximal ideals $M_1$ and $M_2$ of $A$, $M_1 \cap Z(A)=M_2\cap Z(A)$ implies  $M_1=M_2$.   
\end{itemize}
\end{definition}
Note that if $A$ is unital then the above definition agrees with the standard notion of weak centrality.

\begin{remark}\label{remmod}
Since all modular maximal ideals of a $C^*$-algebra $A$ are primitive and since each modular primitive ideal of $A$ is contained in a modular maximal ideal of $A$, the condition (a) in Definition \ref{defwc} can be restated as: 
\begin{itemize}
\item[{\rm (a')}] No modular primitive ideal of $A$ contains $Z(A)$.
\end{itemize}
\end{remark}

The justification of Definition \ref{defwc} will be given in the following series of results. First consider one example.

\begin{example}
Let $X$ be a compact Hausdorff space, $\H$ a separable infinite-dimensional Hilbert space and $A:=C(X,\KK(\H))$. Then each primitive ideal of $A$ is of the form $P_t:=\{f \in A : \, f(t)=0\}$ for some $t \in X$. Since $A/P_t \cong \KK(\H)$ for all $t \in X$, all primitive ideals of $A$ are maximal and non-modular. It follows from Proposition \ref{propzerocen} that $Z(A)=\{0\}$ and $A$ has the CQ-property. Secondly, $\M(A)=\emptyset$ so that $A$ is trivially weakly central even though $P_t \cap Z(A) =\{0\}$ for all of the maximal ideals $P_t$. On the other hand, $A^\sharp$ can be identified with the $C^*$-subalgebra of $B:=C(X,\BB(\H))$ that consists of all functions $f \in B$ for which there exists a scalar $\lm$ such that $f(t)-\lm 1 \in \KK(\H)$ for all $t \in X$. Then $A$ is the unique maximal ideal of  $A^\sharp$ and hence $A^\sharp$ is weakly central.
\end{example}

\begin{proposition}\label{propwc}
For a non-unital $C^*$-algebra $A$ the following conditions are equivalent:
\begin{itemize}
\item[{\rm (i)}] $A$ is weakly central.
\item[{\rm (ii)}] $A^\sharp$ is weakly central. 
\end{itemize}
\end{proposition}
\begin{proof}
(i) $\Longrightarrow$ (ii). Assume that $A$ is weakly central and let $M_1,M_2 \in \M(A^\sharp)$ such that $M_1 \cap Z(A^\sharp)=M_2 \cap Z(A^\sharp)$.

If one of $M_1$ or $M_2$ is $A$, so is the other. Indeed, assume for example that  $M_1=A$. If $M_2 \neq A$ then, by Lemma \ref{lemmaxJ}, $M_2 \cap A$ is a modular maximal ideal of $A$. We have
       $Z(A) = A \cap Z(A^{\sharp}) = M_2 \cap Z(A^{\sharp})$
and so $Z(A)$ is contained in $M_2 \cap A$, contradicting the weak centrality of $A$.

Therefore assume that both $M_1$ and $M_2$ are not $A$. Again, by Lemma  \ref{lemmaxJ}, $N_1:= M_1 \cap A$ and $N_2:=M_2 \cap
A$ are modular maximal ideals of $A$ such that
$N_1 \cap Z(A) =N_2 \cap Z(A)$. The weak centrality of $A$ forces $N_1=N_2$, so Lemma \ref{lemmaxJ} implies $M_1=M_2$. Hence, $A^\sharp$ is weakly central.

\smallskip

(ii) $\Longrightarrow$ (i). Suppose that $A^{\sharp}$ is weakly central.
Let $M \in \M(A)$. By Lemma \ref{lemmaxJ}, there exists $N \in \M(A^{\sharp})$ such that $M = N \cap A$. Since $N \neq A$, it follows that
 $$N \cap Z(A^{\sharp}) \neq A \cap Z(A^{\sharp}) = Z(A).$$
But $Z(A)$ is a maximal ideal in $Z(A^{\sharp})$. Thus $Z(A)$ is not contained in
$N$ and consequently neither in $M$.
    
Now suppose that $M_1, M_2 \in \M(A)$ and $M_1 \cap Z(A) = M_2 \cap Z(A)$. By Lemma \ref{lemmaxJ}, there exist $N_1, N_2 \in \M(A^{\sharp})$ such that $M_1 = N_1 \cap A$ and $M_2 = N_2 \cap A$.  We have
$$(N_1 \cap Z(A^{\sharp})) \cap Z(A) = M_1 \cap Z(A) = (N_2 \cap Z(A^{\sharp})) \cap Z(A).
$$
By the previous paragraph, $M_1$ and $M_2$ do not contain $Z(A)$. It follows that the maximal ideals $N_1 \cap Z(A^{\sharp})$ and $N_2 \cap Z(A^{\sharp})$ of $Z(A^{\sharp})$ do not contain $Z(A)$ and hence must be equal by Lemma \ref{lemmaxJ} (applied to the ideal $Z(A)$ of $Z(A^{\sharp})$). By the weak centrality of $A^{\sharp}$, we have $N_1= N_2$ and hence $M_1 = M_2$. Thus $A$ is weakly central.
\end{proof}

As a direct consequence of Vesterstr\o m's theorem (Theorem \ref{VT}) and Propositions \ref{propwc} and \ref{propcqunit} we get the next characterisation.
\begin{corollary}\label{corCQWC}
For a $C^*$-algebra $A$ the following conditions are equivalent:
\begin{itemize}
\item[{\rm (i)}] $A$ has the CQ-property.
\item[{\rm (ii)}] $A$ is weakly central.
\end{itemize}
\end{corollary}

\begin{remark}\label{remclosed}
An immediate consequence of Proposition \ref{idquCQ} and Corollary \ref{corCQWC} is that the class of weakly central $C^*$-algebras is closed under forming ideals and quotients. 
\end{remark}

The next simple fact follows directly from Corollary \ref{corCQWC}, Remark \ref{remqcid} (d) and Remark \ref{remmod}.

\begin{proposition}\label{propmodid}
For a $C^*$-algebra $A$ the following conditions are equivalent:
\begin{itemize}
\item[{\rm (i)}] $A$ is quasi-central and weakly central.
\item[{\rm (ii)}] $A$ has the CQ-property and every primitive ideal of $A$ is modular.
\end{itemize}
\end{proposition}

Part (b) of the next result overlaps with \cite[Corollary 2.13]{ART}, but the proof here avoids the use of a composition series.

\begin{corollary}\label{cor:postDix}
Let $A$ be a postliminal $C^*$-algebra.
\begin{itemize}
\item[(a)] $A$ has the (singleton) Dixmier property if and only if $A$ is weakly central.
\item[(b)] If every irreducible representation of $A$ is infinite-dimensional, then $A$ has the CQ-property and the (singleton) Dixmier property and is weakly central.
\end{itemize}
\end{corollary}

Before proving Corollary \ref{cor:postDix} we record the next simple fact which will be also used in Section \ref{S4}.
\begin{remark}\label{rem:postlimmax}
Let $A$ be a postliminal $C^*$-algebra. If $\M(A)\neq \emptyset$, then for each $M \in \M(A)$, $A/M$ is a unital simple postliminal $C^*$-algebra and  thus $A/M\cong \MM_n(\C)$ for some $n \in \N$.
\end{remark}

\begin{proof}[Proof of Corollary \ref{cor:postDix}]
(a) By  \cite[Theorem~2.12]{ART} a postliminal $C^*$-algebra $A$ has the (singleton) Dixmier property if and only if it has the CQ-property. It remains to apply Corollary \ref{corCQWC}.

(b) Assume that $A$ contains a modular maximal ideal $M$. By Remark \ref{rem:postlimmax} $A/M \cong \MM_n(\C)$ for some $n\in \N$. Therefore, $A$ has a finite-dimensional irreducible representation; a contradiction. Thus all primitive ideals of $A$ are non-modular (Remark \ref{remmod}), so by Proposition \ref{propzerocen} $A$ has the CQ-property. The other properties follow from \cite[Theorem~2.12]{ART} and Corollary \ref{corCQWC}.
\end{proof}

\smallskip

For the main results of this section (Theorems \ref{chQC} and \ref{thmJcq}), we shall need to consider the following subsets of $\M(A)$ for an arbitrary $C^*$-algebra $A$:
\begin{itemize}
\item[-] $T^1_A$ as the set of all $M \in \M(A)$ such that $Z(A)\subseteq M$.
\item[-] $T^2_A$ as the set of all $M \in \M(A)$ for which exists $N \in \M(A)$ such that $M\neq N$, $Z(A)\nsubseteq M, N$ and $M \cap Z(A)=N \cap Z(A)$.
\item[-] $T_A:=T_A^1 \cup T_A^2$. 
\end{itemize}

The set $T^1_A$ is obviously closed in $\M(A)$. The next example shows that this is not generally true for the set $T_A^2$ (and consequently $T_A$).

\begin{example}\label{ex2sub2}
Let
$A$ be the $C^*$-algebra consisting of all functions $a \in C([0,1],\MM_2(\C))$ which are diagonal at $1/n$ for all $n\in \N$ and scalar at zero. Then $A$ is a unital $2$-subhomogeneous $C^*$-algebra, so $\M(A)=\P(A)$ and
$$T_A=T_A^2=\{P \in \P(A):\, \exists Q \in \P(A), \, P\neq Q,\, P\cap Z(A)=Q\cap Z(A)\}.$$
Let $\lambda_n(a),$
$\mu_n(a)$ and $\eta(a)$ be complex numbers such that 
$$a(1/n)=\di(\lambda_n(a),\mu_n(a)) \ \ (n \in \N) \quad \mbox{and} \quad a(0)=\di(\eta(a),\eta(a)).$$  
If we denote by $\lambda_n$, $\mu_n$ and  $\eta$ the $1$-dimensional (irreducible) representations of $A$ defined respectively by the assignments $a \mapsto \lambda_n(a)$, $a \mapsto \mu_n(a)$ and $a \mapsto \eta(a)$, it is easy to verify that
$$T_A=\{\ker \lambda_n : \, n \in \N\}\cup \{\ker \mu_n : \, n \in \N\}.$$
Then $\ker T_A$  consists of all functions in $A$ which vanish at $1/n$ $(n\in \N)$, and hence vanish at $0$ too.
Therefore, $\ker \eta \in \overline{T_A} \setminus  T_A$, so $T_A$ is not closed in $\M(A)=\P(A)$.
\end{example}

\begin{lemma}\label{lem:cqcont}
If $A$ is a $C^*$-algebra, then $\ker T_A$ contains any weakly central ideal of $A$.
\end{lemma}
\begin{proof}
Let $J$ be a weakly central ideal of $A$. Suppose that $M \in T^1_A$, so that $Z(A)\subseteq M$. Then $J \subseteq M$, for otherwise $M \cap J \in \M(J)$ (Lemma \ref{lemmaxJ}) and $M \cap J$ contains $Z(J)$, contradicting the weak centrality of $J$.  Secondly, suppose that $M \in T^2_A$. Then there exists $N\in \M(A)$ such that $N \neq M$, $Z(A)\nsubseteq M,N$ and $M \cap Z(A)=N\cap Z(A)$.
We have
\begin{equation}\label{eq:intcent}
(M \cap J) \cap Z(J) = (M \cap Z(A)) \cap J = (N \cap J) \cap Z(J). 
\end{equation}
Suppose that $M$ does not contain $J$. Then $M \cap J \in \M(J)$ (Lemma \ref{lemmaxJ}) and $M \cap J$ does not contain $Z(J)$ by the weak centrality of $J$.  By (\ref{eq:intcent}), $N \cap J$ does not contain $Z(J)$ and hence $N$ does not contain $J$. Since $J$ is weakly central, it follows from (\ref{eq:intcent}) that $M \cap J = N \cap J$ and hence (again by Lemma \ref{lemmaxJ}) $M=N$; a contradiction. Thus $J \subseteq M$ as required.
\end{proof}

Given a $C^*$-algebra $A$ we also define 
$$S_A:=\{P \in \P(A) : \, P \mbox{ is non-modular}\} \qquad \mbox{and} \qquad J_A:=\ker S_A.$$
By Remarks \ref{nmod} and \ref{remqcid} (c), $J_A$ contains the largest quasi-central ideal $K_A$ of $A$, so in particular $Z(J_A)=Z(A)$. Example \ref{ex2subnqc} shows that $J_A$ can strictly contain $K_A$ (in this case $J_A=A$).

\begin{theorem}\label{chQC}
For a $C^*$-algebra $A$ the following conditions are equivalent:
\begin{itemize}
\item[(i)]  $A$ has the CQ-property. 
\item[(ii)] $A$ is weakly central.
\item[(iii)] $S_A$ is a closed subset of $\P(A)$ and $J_A=K_A$ is a quasi-central weakly central $C^*$-algebra.  
\item[(iv)] There is a weakly central ideal $J$ of $A$ such that all primitive ideals of $A$ that contain $J$ are non-modular.
\item[(v)] There is an ideal $J$ of $A$ such that both $J$ and $A/J$ have the CQ-property and $Z(A/J)=(Z(A)+J)/J$.
\end{itemize}
\end{theorem}

In the proof of implication (v) $\Longrightarrow$ (i) of Theorem \ref{chQC}
we shall use the next simple fact.

\begin{lemma}\label{lemquot}
Let $A$ be a $C^*$-algebra and $J$ an ideal of $A$ such that $A/J$ has the CQ-property and $Z(A/J)=(Z(A)+J)/J$. Then $Z(A/I)=(Z(A)+I)/I$ for any ideal $I$ of $A$ that contains $J$.
\end{lemma}
\begin{proof}
Let $a \in A$ and suppose that $a+I \in Z(A/I)$. Let $\phi: A/I \to (A/J)/(I/J)$ be the canonical isomorphism. Then
$$\phi(a+I) \in Z((A/J)/(I/J))=\frac{Z(A/J)+ I/J}{I/J}=\frac{(Z(A)+J)/J + I/J}{I/J},$$
so there exists $z \in Z(A)$ such that $\phi(a+I) =\phi(z+I)$.
Hence $a+I = z+I$ and so $a \in Z(A)+I$, as required.
\end{proof}

\begin{proof}[Proof of Theorem \ref{chQC}]
(i) $\Longleftrightarrow$ (ii). This is Corollary \ref{corCQWC}.

\smallskip

(ii) $\Longrightarrow$ (iii). Assume that $A$ is weakly central. Let $P\in \P(A)$ be in the closure of $S_A$ in $\P(A)$, that is  $J_A \subseteq P$. Then $Z(A)=Z(J_A)\subseteq P$. Since $A$ is weakly central, $P$ must be non-modular (Remark \ref{remmod}), so $P \in S_A$. Therefore $S_A$ is closed in $\P(A)$. 

Since $J_A$ is an ideal of $A$ and $A$ is weakly central, so is $J_A$ by Remark \ref{remclosed}. It remains to show that $J_A=K_A$. Since $K_A \subseteq J_A$, it suffices to show that $J_A$ is quasi-central. Assume there exists $R\in \P(J_A)$ that contains $Z(A)=Z(J_A)$ and let $P\in \P_{J_A}(A)$ such that  $R=P \cap J_A$. Obviously $Z(A)\subseteq P$. Since $S_A$ is closed in $\P(A)$, the set $\P_{J_A}(A)$ consists of all modular primitive ideals of $A$. In particular, $P$ is a modular primitive ideal of $A$ that contains $Z(A)$, which  (together with Remark \ref{remmod}) contradicts the weak centrality of $A$.

\smallskip

(iii) $\Longrightarrow$ (iv). Choose $J=J_A=K_A$.

\smallskip

(iv) $\Longrightarrow$ (v). Let $J$ be a weakly central ideal of $A$ such that all primitive ideals in $\P^J(A)$ are non-modular. By Corollary \ref{corCQWC} $J$ has the CQ-property. Also, all primitive ideals of $A/J$ are non-modular, so by  Proposition \ref{propzerocen} $Z(A/J)=\{0\}$ and $A/J$ has the CQ-property.  Further, since $J=\ker \P^J(A)$ and $Z(A)$ is contained in each $P \in \P^J(A)$ (Remark \ref{nmod}), $Z(A)\subseteq J$. Thus $$(Z(A)+J)/J=\{0\}=Z(A/J).$$

\smallskip

(v) $\Longrightarrow$ (i). Assume that $A$ does not have the CQ-property. By Corollary \ref{corCQWC} this is equivalent to say that $A$ is not weakly central. Since $J$ has the CQ-property, it is weakly central (Corollary \ref{corCQWC}), so by Lemma \ref{lem:cqcont} $J$ is contained in $\ker T_A$. We have the following two possibilities.

\emph{Case 1}. There is $M \in \M(A)$ such that $Z(A)\subseteq M$. Then $M \in T_A^1$ so $J \subseteq \ker T_A \subseteq M$. Thus, by Lemma \ref{lemquot},
$$\C \cong Z(A/M)=(Z(A)+M)/M=\{0\};$$
a contradiction.

\emph{Case 2}. There are distinct $M,N \in \M(A)$ such that $Z(A)\nsubseteq M,N$ and $M \cap Z(A) =N\cap Z(A)$. Then 
$$Z\left(\frac{A}{M\cap N}\right)\cong Z(A/M) \oplus Z(A/N) \cong \C \oplus \C.$$
On the other hand, $M,N \in T_A^2$, so $J \subseteq \ker T_A \subseteq M \cap N$. 
Since $Z(A) \nsubseteq M$, $M \cap Z(A)$ is a maximal ideal of $Z(A)$, so using Lemma \ref{lemquot} we get  
$$
Z\left(\frac{A}{M\cap N}\right)=\frac{Z(A)+(M \cap N)}{M \cap N}\cong \frac{Z(A)}{(M \cap N) \cap Z(A)} = \frac{Z(A)}{M \cap Z(A)} \cong \C;
$$
a contradiction.
\end{proof}

Recall that a $C^*$-algebra $A$ is said to be \emph{central} if $A$ is quasi-central and for all $P_1,P_2 \in \P(A)$, $P_1 \cap Z(A)=P_2 \cap Z(A)$ implies $P_1=P_2$ (see \cite[Section~9]{Kap1}). 
\begin{remark}\label{remcent}
It is well-known that a quasi-central $C^*$-algebra $A$ is central if and only if $\P(A)$ is a Hausdorff space (see e.g. \cite[Proposition~3]{Del1}). In particular, by Example \ref{exqc} (b), all homogeneous $C^*$-algebras are central. Further, all central $C^*$-algebras are obviously weakly central.
\end{remark}

\begin{corollary}\label{liminalCQ}
A liminal $C^*$-algebra $A$ has the CQ-property if and only if the set of all
kernels of infinite-dimensional irreducible representations of $A$ is closed in $\P(A)$ and the intersection of these kernels is a central $C^*$-algebra.
\end{corollary}
\begin{proof}
Since $A$ is liminal, an irreducible representation of $A$ is infinite-dimensional if and only if its kernel is a non-modular primitive ideal of $A$.  Thus 
$$S_A=\{\ker \pi : \, [\pi] \in \hat{A}, \, \pi \mbox { infinite-dimensional}\},$$
where $\hat{A}$ denotes the spectrum of $A$. Hence, by Theorem \ref{chQC}, $A$ has the CQ-property if and only if $S_A$ is closed in $\P(A)$ and $J_A=\ker S_A$ is a quasi-central weakly central $C^*$-algebra. Suppose that $S_A$ is closed in $\P(A)$. Then all irreducible representations of $J_A$ are finite-dimensional. In particular, all primitive ideals of $J_A$ are modular and maximal, so weak centrality and quasi-centrality of $J_A$ in this case is equivalent to  centrality.      
\end{proof}

We also record the following special case of Corollary \ref{liminalCQ}.

\begin{corollary}\label{corfinwc}
If all irreducible representations of a $C^*$-algebra $A$ are finite-dimensional, then $A$ has the CQ-property if and only if $A$ is central.
\end{corollary}

\begin{remark}
In contrast to Proposition \ref{propwc} the multiplier algebras of weakly central $C^*$-algebras do not have to be weakly central. In fact, Somerset and the first-named author exhibited an example of a homogeneous (hence central) $C^*$-algebra $A$ such that $\P(M(A))$ is not Hausdorff \cite[Theorem~1]{AS1}. Specifically,
 the subhomogeneous $C^*$-algebra $M(A)$ (see e.g. \cite[Proposition~IV.1.4.6]{Bla}) is not (weakly) central.
\end{remark}

\smallskip

It is possible to show that every $C^*$-algebra contains a largest ideal with the CQ-property by using Zorn's lemma and the fact that the sum of two ideals with the CQ-property has the CQ-property. However, in view of Corollary \ref{corCQWC}, we are able to take a different approach that has the merit of obtaining a formula for this ideal in terms of the set $T_A$ of those modular maximal ideals of $A$ which witness the failure of the weak centrality of $A$.

\begin{theorem}\label{thmJcq}
Let $A$ be a $C^*$-algebra. Then $\ker T_A$ is the largest weakly central ideal of $A$. 
\end{theorem}
\begin{proof}
Set $J:=\ker T_A$. By Lemma \ref{lem:cqcont} it suffices to prove that $J$ is weakly central. For this, we begin by assuming that $A$ is unital (so that $T_A=T_A^2$) and that $J$ is not weakly central. We have two possibilities.

\emph{Case 1.} There is $M_0 \in \M(J)$ such that $Z(J)\subseteq M_0$. By Lemma \ref{lemmaxJ} there exists $N_0 \in \M_J(A)$ such that $M_0=N_0 \cap J$. Since
$$\ker\{N \cap Z(A) :\, N \in T_A\}=Z(J)\subseteq N_0 \cap Z(A)\in \M(Z(A)),$$
there is a net $(N_\a)$ in $T_A$ such that 
\begin{equation}\label{eqwc1}
\lim_\a N_\a \cap Z(A) = N_0 \cap Z(A).
\end{equation}
in $\M(Z(A))$. Since $A$ is unital, $\M(A)$ is a compact subspace of $\P(A)$, so there is a subnet $(N_{\a(\beta)})$ of $(N_\a)$ that converges to some $N_0' \in \M(A)$. Then the continuity of the map $\M(A) \to \M(Z(A))$, defined by $M \mapsto M \cap Z(A)$, implies that
\begin{equation}\label{eqwc2}
\lim_\beta N_{\a(\beta)} \cap Z(A)= N_0' \cap Z(A).
\end{equation}
Since $\M(Z(A))$ is Hausdorff, (\ref{eqwc1}) and  (\ref{eqwc2}) imply 
\begin{equation}\label{eqwc3}
N_0 \cap Z(A)= N_0' \cap Z(A).
\end{equation}
Obviously $N_0'$ lies in the closure of $T_A$, so $J\subseteq N_0'$. Since $N_0 \in  \M_J(A)$, $N_0 \neq N_0'$, so (\ref{eqwc3}) implies $N_0,N_0' \in T_A$. In particular, $J \subseteq N_0$; a contradiction.

\smallskip

\emph{Case 2.} There are $M_1,M_2 \in \M(J)$ such that $M_1\neq M_2$, $Z(J)\nsubseteq M_1,M_2$ and $M_1 \cap Z(J)=M_2 \cap Z(J)$. By Lemma \ref{lemmaxJ} there are $N_1,N_2 \in \M_J(A)$ such that $M_1=N_1 \cap J$ and $M_2 =N_2 \cap J$. Since $Z(J)=J\cap Z(A)$ is an ideal of $Z(A)$,  
$$N_1 \cap Z(A), N_2 \cap Z(A) \in \M_{Z(J)}(Z(A))$$
and
$$(N_1 \cap Z(A))\cap Z(J)=M_1 \cap Z(J)= (N_2 \cap Z(A))\cap Z(J),$$
Lemma \ref{lemmaxJ} (applied to $Z(A)$ and its ideal $Z(J)$) implies that $N_1 \cap Z(A)=N_2 \cap Z(A)$. Since $N_1 \neq N_2$, we conclude that $N_1,N_2 \in T_A$, so $J \subseteq N_1 \cap N_2$; a contradiction.

\smallskip

We have now established that $\ker T_A$ is weakly central in the case that $A$ is unital. We suppose next that $A$ is non-unital. Then, by the above arguments, $\ker T_{A^\sharp}$ is a weakly central ideal of $A^\sharp$. Since $\ker T_{A^\sharp} \cap A$ is an ideal of $\ker T_{A^\sharp}$, it is weakly central by Remark \ref{remclosed}. Hence, it suffices to show that 
$$J:=\ker T_A \subseteq \ker T_{A^\sharp} \cap A.$$ 
So let $M \in T_{A^\sharp}$. We only have to show that $M$ contains $J$.
Since $A^\sharp$ is unital, $T_{A^\sharp}=T_{A^\sharp}^2$, so there is $M' \in \M(A^\sharp)$ such that $M' \neq M$ and  $M \cap Z(A^\sharp)=M' \cap Z(A^\sharp)$. We distinguish three possibilities.
\begin{itemize}
\item[-] $M=A$. Then clearly $M$ contains $J$.
\item[-] $M'=A$. Then $M$ does not contain $A$, so by Lemma \ref{lemmaxJ} $M \cap A$ is a modular maximal ideal of $A$ containing $Z(A)$ (since $M'$ does). Therefore, $M \cap A$ is in $T^1_A$ and hence contains $J$.
\item[-] Both $M$ and $M'$ are not $A$. Again, by Lemma \ref{lemmaxJ},  $M \cap A$ and $M' \cap A$ are distinct modular maximal ideals of $A$ having the same intersection with $Z(A)$. So either $M \cap A$ is in $T^1_A$ or it is in $T^2_A$.  In either case $M$ contains $J$.
\end{itemize}
\end{proof}

In the sequel, for any $C^*$-algebra $A$ by $J_{wc}(A)$ we denote the largest weakly central ideal of $A$. By Corollary \ref{corCQWC}, $J_{wc}(A)$ is precisely the largest ideal of $A$ with the CQ-property.

\begin{corollary}
Let $A$ be a $C^*$-algebra.
\begin{itemize}
\item[(a)] For any ideal $I$ of $A$ we have $J_{wc}(I)=I \cap J_{wc}(A)$.
\item[(b)] The sum of any two weakly central ideals of $A$ is a weakly central ideal of $A$.
\end{itemize} 
\end{corollary}
\begin{proof}
(a) Since $J_{wc}(I)$ is a weakly central ideal of $I$, it is also a weakly central ideal of $A$, so  $J_{wc}(I) \subseteq I \cap J_{wc}(A)$.

Conversely, since $I \cap J_{wc}(A)$ is an ideal of $J_{wc}(A)$, it is weakly central by Remark \ref{remclosed}. Hence, $I \cap J_{wc}(A)\subseteq J_{wc}(I)$.

\smallskip

(b) If $I_1$ and $I_2$ are weakly central ideals of $A$,  then by Theorem \ref{thmJcq} both $I_1$ and $I_2$ are contained in $J_{wc}(A)$, so $I_1+I_2 \subseteq J_{wc}(A)$. Thus, $I_1+I_2$ is weakly central by Remark \ref{remclosed}. 
\end{proof}

The next two examples demonstrate that there are non-trivial $C^*$-algebras whose largest weakly central ideal is zero. 

\begin{example}\label{ex:Heis}
Let $A$ be the rotation algebra (the $C^*$-algebra of the discrete three-dimensional Heisenberg group, see \cite{AP} and the references therein). For each $t$ in the unit circle $\mathbb{T}$, there is an ideal $J_t$ of $A$ such that, with $A_t := A/J_t$, $A$ is $*$-isomorphic to a continuous field of $C^*$-algebras $(A_t)_{t\in \mathbb{T}}$ via the assignment $a \mapsto (a+J_t)_{t\in \mathbb{T}}$. This isomorphism maps $Z(A)$ onto $C(\mathbb{T})$. If $t\in \mathbb{T}$ is a root of unity then $A_t$ is a non-simple homogeneous $C^*$-algebra. If $P$ is a primitive ideal of $A$ that contains $J_t$ then
$P \in \M(A)$, $P \cap Z(A) = J_t \cap Z(A) \neq Z(A)$ and hence $P \in T^2_A$.
It follows that $J_{wc}(A)=\ker T_A \subseteq J_t$. Since the roots of unity form a dense subset of $\mathbb{T}$ and the field $(A_t)_{t\in \mathbb{T}}$ is continuous, $J_{wc}(A)=\{0\}$. Consequently, no non-zero ideal of $A$ has the CQ-property. 
\end{example}

\begin{example}\label{ex:freegroup}
Consider the $C^*$-algebra $A=C^*(\mathbb{F}_2)$ (the full $C^*$-algebra of the free group $\mathbb{F}_2$ on two generators). Then by \cite{Cho}, $A$ is a unital primitive residually finite-dimensional $C^*$-algebra (that is, the intersection of the kernels of the finite-dimensional irreducible representations of $A$ is $\{0\}$). As $A$ is unital and primitive, $Z(A)=\C 1$, so $T_A=\M(A)$. In particular, $J_{wc}(A)=\ker T_A$ is contained in the intersection of the kernels of the finite-dimensional irreducible representations of $A$ which is zero. Therefore $J_{wc}(A) = \{0\}$, as with the rotation algebra.
\end{example}

\begin{remark}\label{rem:fin}
Both $C^*$-algebras in Examples \ref{ex:Heis} and \ref{ex:freegroup} are antiliminal. In Example \ref{exRobTh} we shall also give an example of a (separable) continuous-trace $C^*$-algebra $A$ for which $J_{wc}(A)=\{0\}$.

On the other hand, if $A$ is a $C^*$-algebra for which all irreducible representations have finite dimension, then it follows from \cite[Corollary~3.8]{Gog2} that the ideal $J_{wc}(A)$ is essential.
\end{remark}

\smallskip

We record next a slightly surprising result which can be used for a direct argument that the sum of two ideals with the CQ-property has the CQ-property. 

\begin{proposition}
Let $A$ be a $C^*$-algebra and let $J$ and $K$ be ideals of $A$.  If one of $J$ or $K$ has the CQ-property, then $Z(J+K)=Z(J)+Z(K)$.
\end{proposition}
\begin{proof}
Assume that $J$ has the CQ-property and let $z\in Z(J+K)$. Then $z+K \in Z((J+K)/K)$. Let $\phi: (J+K)/K \to J/(J \cap K)$  be the canonical isomorphism. Then, since $J$ has the CQ-property and $J \cap K$ is an ideal of $J$,
$$
\phi(z+K) \in  Z\left(\frac{J}{J\cap K}\right) = \frac{Z(J) + (J \cap K)}{J \cap K}.$$
So there exists $y\in Z(J)$ such that
$$\phi(z+K) = y + (J \cap K) = \phi(y+K).$$
Hence $z+K = y+K$ and so $z-y \in K \cap Z(A) =Z(K)$. It follows that
$Z(J+K) \subseteq Z(J) + Z(K)$.  For the reverse inclusion, observe that
$$(J \cap Z(A)) + (K \cap Z(A)) \subseteq (J+K) \cap Z(A).$$
\end{proof}

The next example shows that if both ideals $J$ and $K$ of a $C^*$-algebra $A$ fail to satisfy the CQ-property, then $Z(J+K)$ can strictly contain $Z(J)+Z(K)$. 
\begin{example}\label{ex:Dix}
Let $\H$ be a separable infinite-dimensional Hilbert space and let $p \in \BB(\H)$ be any projection with infinite-dimensional kernel and image. Set 
$$A:=\KK(\H)+\C p + \C (1-p)\subset \BB(\H)$$ 
\cite[NOTE 1,~p.257]{Dix}. Then $A$ has precisely two maximal ideals, namely 
$$J:=\KK(\H)+\C p \qquad \mbox{and} \qquad K:=\KK(\H)+\C (1-p).$$ Obviously $Z(J)=Z(K)=\{0\}$, but $Z(J+K)=Z(A)=\C1$.

For later use, we also note that
$$
J_{wc}(A) = \ker T_A = J \cap K = \KK(\H)
$$
and hence $A/J_{wc}(A)$ is abelian.
\end{example}

We finish this section with a generalization of \cite[Theorem~3.1]{Arc2} for arbitrary $C^*$-algebras. For $C^*$-algebras $A_1$ and $A_2$, we denote their algebraic tensor product by $A_1 \odot A_2$. If $\beta$ is any $C^*$-norm on $A_1 \odot A_2$, we denote the completion of $A_1 \odot A_2$ with respect to $\beta$ by $A_1 \otimes_\beta A_2$. 

\begin{theorem}\label{thm:tenprod}
Let $A_1$ and $A_2$ be $C^*$-algebras. The following conditions are equivalent:
\begin{itemize}
\item[(i)] Both $A_1$ and $A_2$ have the CQ-property.
\item[(ii)] $A_1 \otimes_\beta A_2$ has the CQ-property for every $C^*$-norm $\beta$.
\item[(iii)] $A_1 \otimes_\beta A_2$ has the CQ-property for some $C^*$-norm $\beta$. 
\end{itemize} 
\end{theorem}
\begin{proof}
(i)  $\Longrightarrow$ (ii). Suppose that $A_1$ and $A_2$ have the CQ-property and that $\beta$ is a $C^*$-norm on $A_1 \odot A_2$. Since $A_i^\sharp \subseteq A_i^{**}$ ($i=1,2$), it follows from \cite[Theorem~2]{Arc3} that there is a $C^*$-norm $\beta'$ on $A_1^\sharp \odot A_2^\sharp$ extending $\beta$ (recall that by our convention $A_i^\sharp=A_i$ if $A_i$ is unital). Since $A_i$ ($i=1,2$) has the CQ-property if and only if $A_i^\sharp$ is weakly central,  by \cite[Theorem~3.1]{Arc2} $A_1^\sharp \otimes_{\beta'} A_2^\sharp$ is weakly central. Hence, $A_1^\sharp \otimes_{\beta'} A_2^\sharp$ has the CQ-property and so does its ideal $A_1 \otimes_{\beta} A_2$ (Proposition \ref{idquCQ}).

\smallskip

(ii)  $\Longrightarrow$ (iii). This is trivial.

\smallskip

(iii)  $\Longrightarrow$ (i). Assume that $A_1 \otimes_\beta A_2$ has the CQ-property for some $C^*$-norm $\beta$. Since the minimal tensor product $A_1 \otimes_{\min} A_2$ is $*$-isomorphic to a quotient of $A_1 \otimes_\beta A_2$, it follows from Proposition  \ref{idquCQ} that $A_1 \otimes_{\min} A_2$ has the CQ-property. We show that $A_1$ has the CQ-property (a similar argument applies to $A_2$). So let $I$ be an ideal of $A_1$ and let 
$q_I : A_1\to A_1/I$ be the canonical map. Using the canonical $*$-epimorphism 
$$
q_I \otimes \mathrm{id}_{A_2} : A_1 \otimes_{\min} A_2 \to (A/I)\otimes_{\min} A_2
$$
and two applications of \cite[Corollary~1]{HW}, we have
\begin{eqnarray}\label{eqten}
Z(A_1/I)\otimes Z(A_2)&=& Z((A_1/I) \otimes_{\min} A_2)= (q_I \otimes \mathrm{id}_{A_2}) (Z(A_1)\otimes Z(A_2)) \nonumber \\
&=& q_I(Z(A_1))\otimes Z(A_2).
\end{eqnarray}
Assume that $q_I(Z(A_1))$ is strictly contained in $Z(A_1/I)$. Then there is a non-zero functional $\varphi \in Z(A_1/I)^*$ that annihilates $q_I(Z(A_1))$. If $\psi$ is any non-zero functional on $Z(A_2)$, then $\varphi \otimes \psi$ is a non-zero functional on  $Z(A_1/I) \otimes Z(A_2)$ that annihilates  $q_I(Z(A_1))\otimes Z(A_2)$, contradicting (\ref{eqten}). Thus
$$Z(A_1/I)=q_I(Z(A_1))=(Z(A_1)+I)/I$$
as desired.

\end{proof}

\section{CQ-elements in $C^*$-algebras}\label{S4}

Motivated by \cite{BG}, and with a view to identifying the individual elements which prevent the CQ-property, we now introduce a local version of the CQ-property.
 
\begin{definition}\label{def:cqel}
Let $A$ be a $C^*$-algebra. We say that an element $a \in A$ is a \emph{CQ-element} of $A$ if for every ideal $I$ of $A$, $a+I \in Z(A/I)$ implies $a \in Z(A)+I$.
\end{definition}

By $\cq(A)$ we denote the set of all CQ-elements of $A$. Obviously $A$ has the CQ-property if and only if $\cq(A)=A$. We also define $V_A:=A \setminus \cq(A)$, which is the set of elements which prevent $A$ from having the CQ-property.

We state the following $C^*$-algebraic version of \cite[Proposition 2.2]{BG}.

\begin{proposition}\label{cqel}
Let $A$ be a  $C^*$-algebra and let $a\in A$. The following conditions are equivalent:
\begin{enumerate}
\item[{\rm (i)}] $a\in \cq(A)$.
\item[{\rm (ii)}] For every $*$-epimorphism $\phi:A\to B$, where $B$ is another $C^*$-algebra, $\phi(a)\in Z(B)$ implies $a\in Z(A) +\ker\phi$.
\item[{\rm (iii)}]  $a\in Z(A) + {\rm Id}_A([a,A])$.
\end{enumerate}
\end{proposition}

\begin{proposition}\label{propcq}
Let $A$ be a $C^*$-algebra. 
\begin{itemize}
\item[(a)] $\cq(A)$ is a self-adjoint subset of $A$ that is closed under scalar multiplication.
\item[(b)] $Z(A)+\cq(A)\subseteq \cq(A)$. 
\item[(c)]  If $I$ is an ideal of $A$ then $\cq(I)=I \cap \cq(A)$. In particular, $I$ has the CQ-property if and only if $I \subseteq \cq(A)$.
\item[(d)] If $A$ is unital, then for any $a \in \cq(A)$ and invertible $x \in A$ we have $xax^{-1}\in \cq(A)$. 
\end{itemize}
\end{proposition}
\begin{proof}
(a) This is trivial.

\smallskip

(b) Let $a \in \cq(A)$ and $z \in Z(A)$. By Proposition \ref{cqel}, $a \in  Z(A) + {\rm Id}_A([a,A])$, so 
$$z+a \in  Z(A) + {\rm Id}_A([a,A])=Z(A)+{\rm Id}_A([z+a,A]).$$
Using again Proposition \ref{cqel} it follows that $z+a \in \cq(A)$.

\smallskip

(c) Let $a \in \cq(I)$. By Proposition \ref{cqel}, $a \in Z(I)+ {\rm Id}_I([a,I])$. Since $Z(I)=I\cap Z(A) \subseteq Z(A)$ and ${\rm Id}_I([a,I]) \subseteq {\rm Id}_A([a,A])$, we get $a \in Z(A) + {\rm Id}_A([a,A])$.
Therefore $a \in \cq(A)$, so $\cq(I)\subseteq I \cap \cq(A)$.

Conversely, let $a \in I \cap \cq(A)$ and $\ep>0$. By Proposition \ref{cqel} there is a finite number of elements
$u_i,v_i,x_i \in A$ and $z \in Z(A)$ such that
\begin{equation}\label{eqest1}
\left\|a-z-\sum_i u_i [a,x_i]v_i\right\|< \frac{\ep}{3}.
\end{equation}
In particular, $\|z+I\|< \ep/3$, so using the canonical isomorphism $(Z(A)+I)/I \cong Z(A)/(I \cap Z(A))$ we can find an element $z' \in I \cap Z(A)=Z(I)$ such that
\begin{equation}\label{eqest2}
 \|z-z'\|< \frac{\ep}{3}.
\end{equation}
Let $(e_\a)$ be an approximate identity for $I$. Then for all indices $i$ 
\begin{equation}\label{eqest3}
\lim_\a u_ie_\a[a,e_\a x_i]e_\a v_i =u_i[a,x_i]v_i.
\end{equation} 
Indeed, since the multiplication on $A$ is continuous, it suffices to show that for any $x \in A$,
$$\lim_\a e_\a[a,e_\a x]e_\a =[a,x].$$
But this follows directly from the estimate 
\begin{eqnarray*}
\|e_\a[a,e_\a x]e_\a -[a,x]\|&\leq& \|e_\a ([a,e_\a x]-[a,x])e_\a\|+ \|e_\a([a,x]e_\a-[a,x])\|+ \|e_\a[a,x]-[a,x]\| \\
&\leq & \|[a,e_\a x]-[a,x]\|+ \|[a,x]e_\a - [a,x]\|+ \|e_\a [a,x]-[a,x]\|.
\end{eqnarray*}
Hence, by (\ref{eqest3}) there are $u_i',v_i',x_i'\in I$ such that 
\begin{equation}\label{eqest4}
\left\|\sum_i u_i [a,x_i]v_i- \sum_i u_i'[a,x_i']v_i'\right\|< \frac{\ep}{3}.
\end{equation}
Then by (\ref{eqest1}), (\ref{eqest2}) and (\ref{eqest4})
$$\left\|a-z'-\sum_i u_i' [a,x_i']v_i'\right\|< \ep.$$
Invoking again Proposition \ref{cqel}, we conclude that $a \in \cq(I)$, so $I \cap \cq(A)\subseteq \cq(I)$.

\smallskip

(d) Assume that $A$ is unital, $a \in \cq(A)$ and $x \in A$ invertible. If $I$ is an arbitrary ideal of $A$ such that $xax^{-1}+I \in Z(A/I)$, then $a+I\in Z(A/I)$. Since $a \in \cq(A)$, this implies $a \in Z(A)+I$. Then also $xax^{-1}\in Z(A)+I$, so $xax^{-1}\in \cq(A)$.
\end{proof}

\begin{corollary}\label{corsumsub}
If $A$ is a $C^*$-algebra, then $Z(A)+J_{wc}(A) \subseteq \cq(A)$.
\end{corollary}
\begin{proof}
By Theorem \ref{thmJcq} and Corollary \ref{corCQWC} $J_{wc}(A)=\ker T_A$ has the CQ-property, so by Proposition \ref{propcq} (c), $J_{wc}(A) \subseteq \cq(A)$. It remains to apply Proposition \ref{propcq} (b).
\end{proof}

\begin{proposition}\label{prop:commutators}
Let $A$ be a $C^*$-algebra.
\begin{itemize}
\item[(a)] All commutators $[a,b]$ ($a,b \in A$) belong to $\cq(A)$.  In particular, $\cq(A)=Z(A)$ if and only if $A$ is abelian.
\item[(b)] All quasi-nilpotent elements of $A$ belong to $\cq(A)$. Moreover if $a \in A$ is quasi-nilpotent, then $ab, ba \in \cq(A)$ for any $b \in A$.
\end{itemize}
\end{proposition}

\begin{proof}
If $x \in A$ is a commutator, quasi-nilpotent, or a product by a quasi-nilpotent element,  we claim that for any primitive ideal $P$ of $A$, $x+P \in Z(A/P)$ implies $x \in P$. It then follows that $x \in \cq(A)$. Indeed, assume that $I$ is an ideal of $A$ such that $x+I \in Z(A/I)$. Then $x+P \in Z(A/P)$ for any $P\in \P^I(A)$, so $x \in P$. As $\ker \P^I(A)=I$, it follows that $x \in I$ and thus $x \in \cq(A)$ as claimed.

So assume that $P$ is a  primitive ideal of $A$ such that $x+P \in Z(A/P)$. If $P$ is non-modular, then $Z(A/P)=\{0\}$, so trivially $x \in P$. Hence, assume that $P$ is modular, so that $Z(A/P)\cong \C$. Then there is a scalar $\lambda$ such that 
\begin{equation}\label{eq:primscalar}
x+P=\lambda 1_{A/P}. 
\end{equation}

(a) Assume that $x$ is a commutator, so that $x=[a,b]$ for some $a,b \in A$. Then by (\ref{eq:primscalar}),
$$[a+P,b+P]=x+P=\lambda 1_{A/P}.$$
As $A/P$ is a unital $C^*$-algebra, it is well-known that this is only possible if  $\lambda=0$. Thus $x\in P$ as claimed.

If $A$ is non-abelian, then there are $a,b \in A$ such that $x:=[a,b]\neq 0$.
Then there is $P \in \P(A)$ such that $x \notin P$, so by the above arguments $x+P\notin Z(A/P)$. In particular, $x \in \cq(A)\setminus Z(A)$.

\smallskip

(b) If $x$ is quasi-nilpotent, so is  $x+P$, since by the spectral radius formula
$$\nu(x+P)=\lim_n \|x^n+P\|^{\frac{1}{n}}\leq \lim_{n} \|x^n\|^{^{\frac{1}{n}}}=\nu(x)=0.$$ 
This together with (\ref{eq:primscalar}) forces $\lambda=0$, so $x \in P$. 

Now assume that $x=ab$, where $a,b \in A$ and  $a$ is quasi-nilpotent.  
As $a+P$ is quasi-nilpotent, it is a topological divisor of zero (see e.g.  \cite[Section~XXIX.4]{GGK}). Hence, there is a sequence of elements $(x_n)$ in $A$ such that 
$$\|x_n+P\|=1 \quad \forall n \in \N \qquad \mbox{and} \qquad  \lim_n \|x_n a+P\|=0.$$
Then, by (\ref{eq:primscalar}), for all $n\in \N$ we have
$$\lambda x_n + P =(x_n+P)(x+P)=(x_na+P)(b +P),$$
so
\begin{equation}\label{eq:topzerodiv}
|\lambda|\leq  \|x_na+P\|\|b+P\|.
\end{equation}
Since the right side in (\ref{eq:topzerodiv}) tends to zero as $n$ tends to infinity, we conclude that $\lambda=0$. Therefore $x\in P$ as claimed. 

Finally, using the facts that $a\in A$ is quasi-nilpotent if and only if $a^*$ is quasi-nilpotent and that $\cq(A)$ is a self-adjoint subset of $A$ (Proposition \ref{propcq} (a)), we also conclude that $ba \in \cq(A)$ for any $b \in A$ and quasi-nilpotent $a \in A$. 
\end{proof}

\begin{remark}
By Proposition \ref{prop:commutators} (a), non-abelian $C^*$-algebras always contain non-central CQ-elements. On the other hand, in a purely algebraic setting, there are examples of non-abelian algebras in which all centrally stable elements are central \cite[Examples~2.5 and 2.6]{BG} (where central stability is the algebraic counterpart of the CQ-property). 
\end{remark}

If a $C^*$-algebra $A$ is unital, the following fundamental result gives a necessary and sufficient condition for a central element of $A/I$ to lift to a central element of $A$. Recall that by $\Psi_{A} : Z(A)\to C(\P(A))$ we denote the Dauns-Hofmann isomorphism.

\begin{theorem}\label{thmlift}
Let $A$ be a unital $C^*$-algebra and let $I$ be an ideal of $A$. Assume that an element $a \in A$ satisfies $a+I \in Z(A/I)$. Then $a \in Z(A)+I$ if and only if
\begin{equation}\label{eqpsi}
\Psi_{A/I}(a+I)(P_1/I)=\Psi_{A/I}(a+I)(P_2/I)
\end{equation}
for all $P_1,P_2 \in \P^I(A)$ such that $P_1\cap Z(A)=P_2 \cap Z(A)$.
\end{theorem}

\begin{proof}
First assume that $a\in Z(A)+I$, so that $a-z \in I$ for some $z \in Z(A)$.
Suppose that $P_1, P_2 \in \P^I(A)$ and that $P_1 \cap Z(A) = P_2 \cap Z(A)$.
For $i=1,2$, there exists $\lambda_i \in \C$ such that
$$
a + P_i = z + P_i = \lambda_i 1 + P_i \quad  (\mbox{in } A/P_i).
$$
Then  $z-\lambda_1 1, z- \lambda_2 1 \in P_1 \cap Z(A)$  and so 
$(\lambda_1 - \lambda_2)1 \in  P_1$. It follows that $\lambda_1 = \lambda_2$ ($=\lambda$, say). Hence
$$(a+I)+P_i/I= \lm  (1+I)+ P_i/I \quad (\mbox{in } (A/I)/(P_i/I))$$
and therefore
$$\Psi_{A/I}(a+I)(P_1/I)=\lm=\Psi_{A/I}(a+I)(P_2/I).$$

\smallskip 

Conversely, assume that the equality (\ref{eqpsi}) holds for all $P_1,P_2 \in \P^I(A)$ such that $P_1\cap Z(A)=P_2 \cap Z(A)$. Since $A$ is unital, $\P(A)$ is compact and so the closed subset $\P^I(A)$ is also compact.
Define a function
$$f\in C(\P^I(A)) \qquad \mbox{by the formula} \qquad f(P):=\Psi_{A/I}(a+I)(P/I).$$
Let $\phi_A : \P(A)\to \G(A)$ be the complete regularization map (see Section \ref{prel}). Since $\P^I(A)$ is a compact subspace of $\P(A)$, $\phi_A$ continuous and $\G(A)$ a compact Hausdorff space, $K:=\phi_A(\P^I(A))$ is a compact (hence closed) subset of $\G(A)$.
Define a function 
$$g: K \to \C \qquad \mbox{by} \qquad g(G):=f(P),$$ where $P$ is any primitive ideal in $\P^I(A)$ such that $\phi_A(P)=G$. Since $f(P_1)=f(P_2)$ for any two $P_1,P_2 \in \P^I(A)$ such that $P_1\cap Z(A)=P_2 \cap Z(A)$ (which is equivalent to $\phi_A(P_1)=\phi_A(P_2)$), $g$ is well-defined. We claim that $g$ is continuous on $K$. Indeed, let $(G_\alpha)$ be an arbitrary net in $K$ that converges to some $G_0 \in K$.
By general topology, it suffices to show that for any subnet $(G_{\alpha(\beta)})$ of $(G_\alpha)$ there is a further subnet $(G_{\alpha(\beta(\gamma))})$ such that $(g(G_{\alpha(\beta(\gamma))})$ converges to $g(G_0)$. For each index $\beta$ choose $P_{\alpha(\beta)} \in \P^I(A)$ such that $\phi_A(P_{\alpha(\beta)})=G_{\alpha(\beta)}$. Then $(P_{\alpha(\beta)})$ is a net in the compact space $\P^I(A)$, so it has a subnet $(P_{\alpha(\beta(\gamma))})$ convergent to some $P_0 \in \P^I(A)$. Since $\phi_A$ is continuous and $\G(A)$ Hausdorff, $G_0=\phi_A(P_0)$. Further, since $f$ is continuous on $\P^I(A)$, $(f(P_{\alpha(\beta(\gamma))})$ converges to $f(P_0)$.
Therefore  
$$\lim_\gamma g(G_{\alpha(\beta(\gamma))})=\lim_\gamma f(P_{\alpha(\beta(\gamma))})=f(P_0)=g(G_0).$$

By the Tietze extension theorem, there exists a continuous function $\tilde{g}\in C (\G(A))$ that extends $g$. Then a function
$$\tilde{f}: \P(A)\to \C \qquad \mbox{defined by} \qquad \tilde{f}:=\tilde{g} \circ \phi_A$$
is continuous, so by the Dauns-Hofmann theorem there is $z \in Z(A)$ such that 
$\Psi_A(z)=\tilde{f}$. Since for any $P \in \P^I(A)$ we have $\tilde{f}(P)=f(P)$, we conclude $a-z \in P$. Thus $a-z \in I$, so $a \in Z(A)+I$ as desired.
\end{proof}

We now describe the set $\cq(A)$ for an arbitrary $C^*$-algebra $A$. It is somewhat easier to describe its complement $V_A$. In order to do this, we introduce the following  sets:
\begin{itemize}
\item[-] $V_A^1$ as the set of all $a \in A$ for which there exists $M \in \M(A)$
such that $Z(A) \subseteq M$ and $a+M$ is a non-zero scalar in $A/M$,
\item[-] $V_A^2$ as the set of all $a \in A$ for which there exist $M_1,M_2 \in \M(A)$ and  scalars $\lm_1 \neq \lm_2$  such that $Z(A)\nsubseteq M_i$, $M_1 \cap Z(A)= M_2 \cap Z(A)$ and $a+M_i=\lm_i 1_{A/M_i}$ ($i=1,2$).
\end{itemize}

\begin{theorem}\label{thmcqel}
If $A$ is a $C^*$-algebra then $V_A=V_A^1 \cup V_A^2$.
\end{theorem}

\begin{proof}
Assume there exists $a \in V_A^1 \setminus V_A$. 
Then $a \in \cq(A)$ and  there is $M\in \M(A)$ such that $Z(A)\subseteq M$ and $a+M$ is a non-zero scalar in $A/M$. In particular, $a+M \in Z(A/M)$, so the CQ-condition implies $a \in Z(A)+M=M$; a contradiction. This shows $V_A^1 \subseteq  V_A$.

Now assume there exists $a \in V_A^2 \setminus V_A$ and let $M_1,M_2 \in \M(A)$ 
and $\lm_1,\lm_2 \in \C$, $\lm_1 \neq \lm_2$, such that $Z(A)\nsubseteq M_i$, $M_1 \cap Z(A)= M_2 \cap Z(A)$ and $a+M_i=\lm_i 1_{A/M_i}$ ($i=1,2$). Then by maximality of $M_1$ and $M_2$ we have $M_1+M_2=A$, so $A/(M_1\cap M_2)\cong (A/M_1) \oplus (A/M_2)$. Hence $a+ (M_1\cap M_2) \in Z(A/(M_1\cap M_2))$. Since $a\in \cq(A)$, this forces $a \in Z(A)+(M_1 \cap M_2)$. Choose a central element $z \in Z(A)$ such that $a-z \in M_1 \cap M_2$. Obviously, $z+M_i=\lm_i 1_{A/M_i}$ ($i=1,2$). On the other hand, under the canonical isomorphisms 
$$\frac{Z(A)+M_1}{M_1}\cong \frac{Z(A)}{M_1 \cap Z(A)}=\frac{Z(A)}{M_2 \cap Z(A)}\cong \frac{Z(A)+M_2}{M_2},$$
$z+M_1$ is mapped to $z+M_2$. This implies $\lm_1=\lm_2$; a contradiction. Therefore $V_A^2 \subseteq V_A$. 

\medskip

Conversely, let $a \in V_A$.

\emph{Case 1}. $A$ is unital.

In this case $V_A^1=\emptyset$. Since $a \notin\cq(A)$,  there exists an ideal $I$  of $A$ such that $a+I \in Z(A/I)$ but $a \notin Z(A)+I$. For each  $P \in \P^I(A)$
set $$\lm_P:=\Psi_{A/I}(a+I)(P/I),$$
where $\Psi_{A/I}: Z(A/I)\to C(\P(A/I))$ is the Dauns-Hofmann isomorphism. 
By Theorem \ref{thmlift} there are $P_1,P_2 \in \P^I(A)$ such that $P_1\cap Z(A)=P_2 \cap Z(A)$ and  $\lm_{P_1} \neq \lm_{P_2}$. Then $a-\lm_{P_i} 1 \in P_i$ ($i=1,2$). Choose maximal ideals $M_1,M_2$ of $A$ such that $P_1 \subseteq M_1$ and $P_2 \subseteq M_2$. Since $A/M_i$ is a quotient of $A/P_i$, it follows that $\lm_{M_i}=\lm_{P_i}$, so $a-\lm_{P_i} 1 \in M_i$ ($i=1,2$). Therefore,  $a+M_1$ and $a+M_2$ are distinct scalars in $A/M_1$ and $A/M_2$, which implies $a \in V_A^2$.

\smallskip

\emph{Case 2}. $A$ is non-unital. 

In this case we  work inside the unitization $A^\sharp$. By Proposition \ref{propcq} (c) $a \in V_{A^\sharp}$. Then, using Case 1, there are maximal ideals $M_1$ and $M_2$ of $A^\sharp$ and scalars $\lm_1 \neq \lm_2$ such that $M_1 \cap Z(A^\sharp) = M_2 \cap Z(A^\sharp)$ and  $a-\lm_i 1\in M_i$ ($i=1,2$). We have two possibilities.

\emph{Case 2.1}. One of $M_1,M_2$ coincides with $A$. Say $M_1=A$. Then $\lm_1=0$ (since $a$ belongs to $A$), so $\lm_2 \neq 0$. In this case $M_2\neq A$, since otherwise $\lm_1=\lm_2=0$; a contradiction. Then, by Lemma \ref{lemmaxJ}, $N_2:=M_2 \cap A$ is a modular maximal ideal of $A$. Since $Z(A)=A \cap Z(A^\sharp)=M_2 \cap Z(A^\sharp)$, $N_2$ contains $Z(A)$. Under the canonical isomorphism
$$\frac{A}{N_2}\cong \frac{A+M_2}{M_2}= \frac{A^\sharp}{M_2},$$
$a+N_2$ is mapped to $a+M_2=\lm_2 1+M_2$, so 
$a+N_2= \lm_2 1_{A/N_2}$. Since $\lm_2 \neq 0$, we conclude that $a$
belongs to $V_A^1$.

\emph{Case 2.2}. $M_1 \neq A$ and $M_2 \neq A$. Then, by Lemma \ref{lemmaxJ}, $N_i:= M_i \cap A$ ($i=1,2$) are modular maximal ideals of $A$ that have the same intersection with $Z(A)$. Similarly as in Case 2.1, using the canonical isomorphisms $A/N_i \cong 
A^\sharp/M_i$, we conclude that $a+N_i = \lm_i 1_{A/N_i}$ ($i=1,2$), which implies
that $a$ belongs to $V_A^2$.

Therefore $a \in V_A^1 \cup V_A^2$, so $V_A=V_A^1 \cup V_A^2$ as claimed.
\end{proof}

\begin{remark}
Theorem \ref{thmcqel} enables us to recapture Corollary \ref{corCQWC} without using Vesterstr\o m's theorem for the unital case (Theorem \ref{VT}). Indeed, if $A$ is weakly central then clearly $V_A$ is empty and so $A$ has the CQ-property. Conversely, suppose that $A$ is not weakly central. If there exists $M \in \M(A)$ such that $Z(A)\subseteq M$ then, taking any $a \in A$ such that $a+M = 1_{A/M}$, we obtain $a \in V^1_A$. Otherwise, there exist distinct $M_1, M_2 \in \M(A)$ such that $M_1 \cap Z(A) = M_2\cap Z(A) \neq Z(A)$. Since $M_1+M_2=A,$ there exists $a\in M_1$ such that $a+M_2 = 1_{A/M_2}$ and hence $a\in V^2_A$. Thus $V_A$ is non-empty and so $A$ does not have the CQ-property.

Also, the methods of this section enable us to give a short alternative proof of the fact that $\ker T_A$ is weakly central (Theorem \ref{thmJcq}). By the preceding paragraph $\ker T_A$ is weakly central if and only if $V_{\ker T_A}=\emptyset$. By Theorem \ref{thmcqel} and Proposition \ref{propcq} (c) it suffices to show that $a \in \ker T_A$ implies $a \notin V_A^1 \cup V_A^2=V_A$. But this is trivial, since for any $M \in \M(A)$ that contains $Z(A)$ we have $M \in T_A^1$, so $a \in M$ and therefore $a+M$ is zero in $A/M$. Similarly, for all 
$M_1, M_2 \in \M(A)$ such that $M_1 \neq M_2$ and $M_1 \cap Z(A) = M_2\cap Z(A) \neq Z(A)$ we have $M_1,M_2 \in T_A^2$, so $a \in M_1 \cap M_2$
and hence $a+M_i$ is zero in $A/M_i$ ($i=1,2$).

\end{remark}

If $A$ is a $C^*$-algebra then by Proposition \ref{prop:commutators} (a) all commutators $[a,b]$ ($a,b \in A$) belong to $\cq(A)$. Let us denote by $[A,A]$ the linear span of all commutators of $A$ and by $\overline{[A,A]}$ its norm-closure. We now characterise when $\cq(A)$ contains  $\overline{[A,A]}$.

\begin{theorem}\label{thm:comtstates}
Let $A$ be a $C^*$-algebra that is not weakly central.
\begin{itemize}
\item[(a)] If for all $M \in T_A$, $A/M$ admits a tracial state then $\overline{[A,A]} \subseteq \cq(A)$. 
\item[(b)] If there is $M \in T_A$ such that $A/M$ does not admit a tracial state, then $[A,A]\nsubseteq \cq(A)$. 
\end{itemize}
\end{theorem}
\begin{proof}
(a) Let $x \in \overline{[A,A]}$. In order to show that $x \in \cq(A)$, it suffices by Theorem \ref{thmcqel} to prove that for each $M\in T_A$, $x+M \in Z(A/M)$ implies $x \in M$. Therefore, fix some $M \in T_A$ and assume that $x+M \in Z(A/M)$, so that  $x+M=\lambda 1_{A/M}$ for some scalar $\lambda$. By assumption $A/M$ admits a tracial state $\tau$. As $x \in \overline{[A,A]}$, clearly $x+M \in \overline{[A/M,A/M]}$. Since $\tau(\overline{[A/M,A/M]})=\{0\}$, we get
$$\lambda = \tau(\lambda 1_{A/M})= \tau(x+M)=0.$$
Thus $x \in M$, as claimed.

(b) Assume that $A/M$ does not admit a tracial state for some $M \in T_A$. 
As $T_A=T_A^1\cup T_A^2$, we have two possibilities.

\emph{Case 1.} $M \in T_A^1$, so that $Z(A)\subseteq M$. By \cite[Theorem~1]{Pop} there is an integer $n>1$, that depends only on $A/M$, such that any element of $A/M$ can be expressed as a sum of $n$ commutators. In particular, there are $a_1, \ldots, a_n, b_1, \ldots , b_n \in A$ such that 
$$\sum_{i=1}^n [a_i,b_i]+M=\sum_{i=1}^n [a_i+M,b_i+M]=1_{A/M},$$ 
so by Theorem \ref{thmcqel}
$$\sum_{i=1}^n [a_i,b_i] \in V_A^1 \subseteq V_A.$$

\emph{Case 2.} $M \in T_A^2$. Then $Z(A)\nsubseteq M$ and there exists $N \in \M(A)$ such that $N \neq M$ and $M \cap Z(A)=N\cap Z(A)$.
By Lemma \ref{lemmaxJ}, $M\cap N$ is a modular maximal ideal of $N$. As $N/(M\cap N) \cong A/M$, $N /(M \cap N)$ also does not admit a tracial state, so by  \cite[Theorem~1]{Pop} there is an integer $n > 1$ and elements $a_1, \ldots, a_n, b_1, \ldots , b_n \in N$ such that   
$$\sum_{i=1}^n [a_i,b_i]+M\cap N=1_{N/(M\cap N)}.$$ 
Using the canonical isomorphism $N/(M\cap N) \cong A/M$, we get
$$\sum_{i=1}^n [a_i,b_i]+M=1_{A/M}$$ 
and thus by Theorem \ref{thmcqel}
$$\sum_{i=1}^n [a_i,b_i] \in V_A^2 \subseteq V_A.$$
\end{proof}

\begin{corollary}\label{cor:postAF}
If $A$ is a postliminal $C^*$-algebra or an AF-algebra, then $\overline{[A,A]} \subseteq \cq(A)$. 
\end{corollary}
\begin{proof}
If $A$ is weakly central, then $\cq(A)=A$ so we have nothing to prove. Hence assume that $A$ is not weakly central, so that there is $M \in T_A$. If $A$ is postliminal then by Remark \ref{rem:postlimmax} $A/M\cong \MM_n(\C)$ for some $n \in \N$, so $A/M$ has a (unique) tracial state. If, on the other hand, $A$ is an AF-algebra, then $A/M$ is a unital simple AF-algebra, so it also admits a tracial state (see e.g. \cite[Proposition~3.4.11]{Lin}). Therefore, the assertion follows directly from Theorem \ref{thm:comtstates} (a).
\end{proof}

By Corollary \ref{corsumsub}, for any $C^*$-algebra $A$,  $\cq(A)$ always contains $Z(A)+J_{wc}(A)$. The next result in particular demonstrates that $\cq(A)$ is a $C^*$-subalgebra of $A$ if and only if $\cq(A)=Z(A)+J_{wc}(A)$. In fact, when this does not hold, $\cq(A)$ fails dramatically to be a $C^*$-algebra.

\begin{theorem}\label{thm:cqequiv}
Let $A$ be a $C^*$-algebra. The following conditions are equivalent:
\begin{itemize}
\item[(i)] $A/J_{wc}(A)$ is abelian.
\item[(ii)] $\cq(A)=Z(A)+J_{wc}(A)$.
\item[(iii)] $\cq(A)$ is closed under addition.
\item[(iv)] $\cq(A)$ is closed under multiplication. 
\item[(v)] $\cq(A)$ is norm-closed.
\end{itemize}
\end{theorem}

\begin{remark}\label{rem:non-ab}
Since $T_A$ is dense in $\P^{J_{wc}(A)}(A)$, it follows from \cite[Proposition~ 3.6.3]{DixB} that $A/J_{wc}(A)$ is non-abelian if and only if there is $M \in T_A$ such that $\dim(A/M)>1$.
\end{remark}

\begin{proof}[Proof of Theorem \ref{thm:cqequiv}]
(i) $\Longrightarrow$ (ii). Assume that $A/J_{wc}(A)$ is abelian. By Corollary \ref{corsumsub} we already know that $Z(A)+J_{wc}(A)\subseteq \cq(A)$, so it suffices to show the reverse inclusion. For any $a \in A$ we have $a+J_{wc}(A)\in A/J_{wc}(A)=Z(A/J_{wc}(A))$,
so if $a \in \cq(A)$, this forces $a \in Z(A)+J_{wc}(A)$. Therefore $\cq(A) = Z(A)+J_{wc}(A)$, as claimed. 

\smallskip

(ii) $\Longrightarrow$ (iii), (iv), (v) is trivial, since $Z(A)+J_{wc}(A)$ is a $C^*$-subalgebra of $A$.

\smallskip 

(iii), (iv) or (v) $\Longrightarrow$ (i). Assume that $A/J_{wc}(A)$ is non-abelian. By Remark \ref{rem:non-ab} there is $M \in T_A$ such that $\dim(A/M)>1$. We  show that $\cq(A)$ is not norm-closed and is neither closed under addition nor closed under multiplication.  As $A/M$ is non-abelian, by \cite[Exercise~4.6.30]{KR} $A/M$ contains a nilpotent element $\dot{q}$ of nilpotency index $2$. By \cite[Proposition~2.8]{AGP} (see also \cite[Theorem~6.7]{OP}), we may lift $\dot{q}$ to a nilpotent element $q \in A$ of the same nilpotency index $2$. As the norm function $\P(A)\ni P \mapsto \|q+P\|$ is lower semi-continuous on $\P(A)$ (see e.g. \cite[Proposition~II.6.5.6 (iii)]{Bla}) and $q \notin M$, the set
\begin{equation}\label{eq:normnilq}
U:=\{P \in \P(A) : \, \|q+P\|>0 \} 
\end{equation} 
is an open neighbourhood of $M$ in $\P(A)$. As $T_A=T_A^1 \cup T_A^2$ we have two possibilities.

\smallskip

\emph{Case 1.} $M \in T_A^1$, so that $Z(A)\subseteq M$. Let $I$ be the ideal of $A$ that corresponds to $U$, so that $U=\P_I(A)$. As $\cq(I)=I \cap \cq(A)$ (Proposition \ref{propcq} (c)), it suffices to show that $\cq(I)$ is not norm-closed and is neither closed under addition nor closed under multiplication. 

By Lemma \ref{lemmaxJ}, $M \cap I$ is a modular maximal ideal of $I$ that  contains $Z(I)=I \cap Z(A)$, so that $M \cap I \in T_I^1$.  Choose a self-adjoint element $a \in I$ such that
\begin{equation}\label{eq:aplusm}
a+(M \cap I)=1_{I/(M\cap I)}.
\end{equation} 
For each non-zero scalar $\mu \in \C$ consider the element 
$$x_\mu:=a\left(a+\mu q\right)\in I.$$ 
We claim that for any $N' \in \M(I)$, $x_\mu+N' \in Z(I/N')$ implies $x_\mu \in N'$, so that $x_\mu \in \cq(I)$ (Theorem \ref{thmcqel}).  Indeed, assume there is $N' \in \M(I)$ such that $x_\mu+N' \in Z(I/N')$. By Lemma \ref{lemmaxJ} there exists $N \in \M_I(A)$ such that $N'=N \cap I$. Then
$$
x_\mu+(N \cap I) = \lambda 1_{I/(N\cap I)}
$$
for some scalar $\lambda$, so using the canonical isomorphism $I/(N\cap I)\cong A/N$ we get 
\begin{equation}\label{eq:xkquot}
(a+N)\left((a+N)+ \mu (q+N)\right)=x_\mu+N=\lambda 1_{A/N}.
\end{equation}
Suppose that $\lambda \neq 0$. Then, by (\ref{eq:xkquot}), the element $a+N$ is right invertible in $A/N$. Since $a+N$ is self-adjoint, it must be invertible in $A/N$.
As $\mu \neq 0$,  (\ref{eq:xkquot}) implies 
\begin{equation}\label{eq:qplusn}
q+N=\frac{1}{\mu}\left(\lambda (a+N)^{-1}-(a+N)\right).
\end{equation} 
The right side in (\ref{eq:qplusn}) defines a normal element of $A/N$, as a linear combination of two commuting self-adjoint elements of $A/N$. Hence, $q+N$ is a normal nilpotent element of $A/N$ which implies $q\in N$.
But as $N \in \M_I(A)$, $N$ belongs to $U$, which contradicts (\ref{eq:normnilq}). Thus $\lambda = 0$ and so $x_\mu \in \cq(I)$ as claimed. 

We claim that 
$$x_{-1}+x_{1} \notin \cq(I) \qquad \mbox{and} \qquad x_{-1}x_1 \notin \cq(I).$$ Indeed, by (\ref{eq:aplusm}) we have
\begin{eqnarray*}
x_{-1}+x_{1} + (M \cap I) &=& (a(a-q)+(M \cap I)) + (a(a+q)+ (M \cap I)) = 2a^2 + (M \cap I)\\
&=&2 1_{I/(M\cap I)}.  
\end{eqnarray*}
Further, since $q^2=0$, we have
\begin{eqnarray*}
x_{-1}x_1+(M\cap I)&=&(a(a-q)+(M\cap I))(a(a+q)+(M\cap I))\\
&=&(1_{I/(M\cap I)}-(q+(M\cap I)))(1_{I/(I\cap M)}+(q+(M\cap I)))\\
&=& 1_{I/(M\cap I)}.
\end{eqnarray*}
Therefore, both $x_{-1}+x_{1}$ and $x_{-1}x_1$ belong to $V_I^1 \subseteq V_I=I\setminus \cq(I)$ (Theorem \ref{thmcqel}), which shows that $\cq(I)$ is neither closed under addition nor closed under multiplication. 

It remains to show that  $\cq(I)$ is not norm-closed. In order to do this, consider the sequence 
$$y_k:=x_{\frac{1}{k}}=a\left(a + \frac{1}{k}q\right) \qquad (k \in \N).$$
Then $(y_k)$ is a sequence in $\cq(I)$ that converges to $a^2$. As  $a^2 + (M \cap I)=1_{I/(M\cap I)}$ (by (\ref{eq:aplusm})),  we conclude that $a^2 \in V_I^1 \subseteq V_I$  (Theorem \ref{thmcqel}), so the proof for this case is finished. 

\smallskip

\emph{Case 2.} $M \in T_A^2$. Then $Z(A)\nsubseteq M$ and there exists $N \in \M(A)$ such that $N \neq M$ and $M \cap Z(A)=N\cap Z(A)$.
As singleton subsets of $\M(A)$ are closed in $\P(A)$, $U':=U\setminus \{N\}$ is also an open  neighbourhood of $M$ in $\P(A)$. Let $J$ be the ideal of $A$ that corresponds to $U'$. Then, by Lemma \ref{lemmaxJ}, $M\cap J \in \M(J)$ and $\dim(J/(M\cap J))=\dim(A/M)>1$. Further, since $N \notin U'$, $J \subseteq N$, so 
$$Z(J)=(N \cap Z(A)) \cap J = (M \cap Z(A)) \cap J.$$
This implies $Z(J)\subseteq M \cap J$ and therefore $M \cap J \in T_{J}^1$. Also, by (\ref{eq:normnilq}), we have trivially $\|q+P\|>0$ for all $P \in U'$. By applying the method of Case 1 to $J$ in place of $I$, we conclude that $\cq(J)$ is not norm-closed and is neither closed under addition nor closed under multiplication. As $\cq(J)=J \cap \cq(A)$ (Proposition \ref{propcq} (c)), the same is true for $\cq(A)$.  
\end{proof}

\begin{remark}
Observe that Theorem \ref{thm:cqequiv} applies to Example \ref{ex:Dix}, giving $\cq(A) =\C 1 +\KK(\H)$.
\end{remark}

\begin{corollary}\label{2-subcq}
If $A$ is a $2$-subhomogeneous $C^*$-algebra, then $\cq(A)=Z(A)+J_{wc}(A)$.
\end{corollary}
\begin{proof}
Since 
$$\{P \in \P(A): \, \dim(A/P)=1\}$$ 
is a closed subset of $\P(A)$ (see e.g. \cite[Proposition~3.6.3]{DixB}), $A$ has a $2$-homogeneous ideal $I$ such that $A/I$ is abelian. Then $I$ is a central $C^*$-algebra (Remark \ref{remcent}) and so $I \subseteq J_{wc}(A)$. Hence $A/J_{wc}(A)$ is abelian and so the result follows from Theorem \ref{thm:cqequiv}.
\end{proof}

In the case when a $C^*$-algebra $A$ is postliminal or an AF-algebra, we also show that the conditions (i)-(v) of Theorem \ref{thm:cqequiv} are equivalent to one additional condition.

\begin{corollary}\label{cor:postlim}
If $A$ is a postliminal $C^*$-algebra or an AF-algebra, then the conditions (i)-(v) of Theorem \ref{thm:cqequiv} are also equivalent to:
\begin{itemize}
\item[(vi)] For any $x \in \cq(A)$, $x^n \in \cq(A)$ for all $n \in \N$. 
\end{itemize}
\end{corollary}

In the proof of Corollary \ref{cor:postlim} we shall use the next two facts. In the sequel we say that a $C^*$-subalgebra $A$ of a unital $C^*$-algebra $B$ is \emph{co-unital} if $A$ contains the identity of $B$.

\begin{lemma}\label{lem:AFcounit}
Let $B$ be a unital simple non-abelian AF-algebra. Then $B$ contains a co-unital finite-dimensional $C^*$-subalgebra with no abelian summand.
\end{lemma}
\begin{proof}
Let $(B_k)_{k \in \N}$ be an increasing sequence of finite-dimensional $C^*$-subalgebras of $B$ such that $1_{B}\in B_k$ for all $k \in \N$ and 
$$B=\overline{\bigcup_{k \in \N} B_k}.$$ 
We claim that there exists $k  \in \N$ such that $B_k$ has no direct summand $*$-isomorphic to $\C$. On a contrary, suppose that for every $k \in \N$, $B_k$ has a direct summand $*$-isomorphic to $\C$ and hence a multiplicative state $\omega_k$. For each $k \in  \N$ let $\psi_k\in \S(B)$ be an extension of $\omega_k$. By the weak$^*$-compactness of $\S(B)$ there exists $\psi \in \S(B)$ and a subnet $(\psi_{k(\alpha)})$ of $(\psi_k)$ such that
$$\psi =w^*-\lim_{\alpha}\psi_{k(\alpha)}.$$
Let $a \in B_{j_0}$ and $b \in B_{k_0}$ for some $j_0,k_0 \in \N$. There exists an index  $\alpha_0$ such that $k(\alpha)\geq \max\{j_0,k_0\}$ for all $\alpha \geq \alpha_0$. Thus 
$$
\psi_{k(\alpha)}(ab)=\psi_{k(\alpha)}(a)\psi_{k(\alpha)}(b)
$$
for all $\alpha \geq \alpha_0$ and so 
\begin{equation}\label{eq:multst}
\psi(ab)=\psi(a)\psi(b).
\end{equation} 
Now suppose that $a \in B_{j_0}$ for some $j_0 \in \N$, that $b \in B$ and that $\varepsilon>0$. Then there exists $k_0 \in \N$ and $b_0 \in B_{k_0}$ such that 
$$\|b-b_0\|< \frac{\varepsilon}{2(1+\|a\|)}.$$
Then
\begin{eqnarray*}
|\psi(ab)-\psi(a)\psi(b)| &\leq & |\psi(ab)-\psi(ab_0)|+ |\psi(a)\psi(b_0)-\psi(a)\psi(b)|\\
&\leq& 2 \|a\|\|b-b_0\|< \varepsilon.
\end{eqnarray*}
Thus, again (\ref{eq:multst}) holds. A similar approximation in the first variable shows that  (\ref{eq:multst}) holds for all $a,b \in B$. Thus, $B$ has a multiplicative state, contradicting the fact that $B$ is simple but not $*$-isomorphic to $\C$.
\end{proof}

\begin{lemma}\label{lem:postmaxqot}
Let $A$ be a $C^*$-algebra such that for some $M \in \M(A)$, $A/M$ contains a co-unital finite-dimensional $C^*$-subalgebra with no abelian summand.
Then there are $a,b \in A$ and an integer $n >1$  such that
\begin{equation}\label{eq:lemlift1}
[a,b]^n+M=1_{A/M}.
\end{equation}
\end{lemma} 
\begin{proof}
Assume that $B$ is a co-unital finite-dimensional $C^*$-subalgebra of $A/M$ with no abelian summand. Then there are  integers $n_1, \ldots , n_k>1$ and a $*$-isomorphism  $\phi : B \to \MM_{n_1}(\C) \oplus \cdots \oplus \MM_{n_k}(\C)$. For each $i=1, \ldots, k$ let $\{\alpha_1^{(i)}, \ldots, \alpha^{(i)}_{n_i}\}$ be the set of all $n_i$-th roots of unity. It is well-known that the set of all commutators of $\MM_{n_i}(\C)$ consists precisely of all matrices of trace zero. Hence, as $\alpha_1^{(i)} + \ldots + \alpha^{(i)}_{n_i} =0$ for all $i=1, \ldots, k$, there are elements $a,b \in A$ such that $a+M,b+M \in B$ and  
$$\phi([a,b]+M)=\di(\alpha^{(1)}_1,\ldots, \alpha^{(1)}_{n_1}) \oplus \cdots \oplus \di(\alpha^{(k)}_1,\ldots, \alpha^{(k)}_{n_k}).$$ 
Let $n$ be the least common multiple of $n_1, \ldots, n_k$. Then 
$$\phi([a,b]+M)^n=1_{\MM_{n_1}(\C)} \oplus \cdots \oplus 1_{\MM_{n_k}(\C)}.$$  Since $B$ is co-unital in $A/M$, this is equivalent to (\ref{eq:lemlift1}).
\end{proof}

\begin{proof}[Proof of Corollary \ref{cor:postlim}]
By Theorem \ref{thm:cqequiv} we only have to prove the implication (vi) $\Longrightarrow$ (i). Assume that (i) does not hold. By Remark \ref{rem:non-ab} there is $M \in T_A$ such that $\dim(A/M)>1$. If $A$ is postliminal or AF (respectively), then $A/M$ is a unital simple $C^*$-algebra that is postliminal or AF (respectively). Hence, by Remark \ref{rem:postlimmax} and Lemma \ref{lem:AFcounit} $A/M$ certainly contains a co-unital finite-dimensional $C^*$-subalgebra with no abelian summand. As $T_A=T_A^1\cup T_A^2$ we have two possibilities.

\emph{Case 1.} $M \in T_A^1$, so that $Z(A)\subseteq M$. By Lemma \ref{lem:postmaxqot}, there are $a,b \in A$ and an integer $n>1$ such that for $x:=[a,b]$ we have $x^n+M=1_{A/M}$. 
In particular $x^n \in V_A^1$ so, by Theorem  \ref{thmcqel}, $x^n \notin \cq(A)$. On the other hand, by Proposition \ref{prop:commutators} (a), $x \in \cq(A)$. 

\smallskip

\emph{Case 2.} $M \in T_A^2$. Then $Z(A)\nsubseteq M$ and there exists $N \in \M(A)$ such that $N \neq M$ and $M \cap Z(A)=N\cap Z(A)$.
By Lemma \ref{lemmaxJ}, $M\cap N$ is a modular maximal ideal of $N$ and $N/(M\cap N) \cong A/M$. 

By Lemma \ref{lem:postmaxqot} (applied to $N$) there are 
$a,b \in N$ and an integer $n>1$ such that for $x:=[a,b]$ we have $x^n+M\cap N=1_{N/(M\cap N)}$.  
Then, using the canonical isomorphism $N/(M\cap N) \cong A/M$, we get 
$x^n+M=1_{A/M}$.  As $x^n \in N$, we have $x^n \in V^2_A$, so $x^n \notin \cq(A)$ by Theorem  \ref{thmcqel}. On the other hand, by Propositions \ref{prop:commutators} (a) and \ref{propcq} (c), $x \in \cq(N)\subseteq \cq(A)$.
\end{proof}

The next example shows that $2$-subhomogeneity in Corollary \ref{2-subcq} cannot be replaced by $n$-subhomogeneity, where $n>2$. It also provides an example of a liminal $C^*$-algebra for which the six equivalent conditions of Corollary \ref{cor:postlim} fail to hold. 
\begin{example}\label{ex3sub}
Let $A$ be the $C^*$-algebra consisting of all functions $a \in C([0,1],\MM_3(\C))$
such that 
$$a(1)=\begin{pmatrix}
\lambda_{11}(a) & \lambda_{12}(a) & 0 \\
\lambda_{21}(a) & \lambda_{22}(a) & 0 \\
0 & 0 & \mu(a)
\end{pmatrix},$$
for some complex numbers $\lambda_{ij}(a), \mu(a)$ ($i,j=1,2$). Then $A$ is a unital $3$-subhomogeneous $C^*$-algebra such that
$$Z(A)=\{\di(f,f,f) : \, f \in C([0,1])\}$$
and
$$T_A=T_A^2=\{\ker \pi, \ker \mu\},$$ 
where $\pi : A \to \MM_2(\C)$ and $\mu : A \to \C$ are irreducible representations of $A$ defined by the assignments
$\pi : a \mapsto (\lambda_{ij}(a))$ and $\mu : a \mapsto \mu(a)$. 
Hence, by Theorem \ref{thmJcq}, 
$$J_{wc}(A)=\ker T_A=\{a \in A: \, a(1)=0\}$$
and so 
$$Z(A)+J_{wc}(A)=\{a \in A : \, a(1) \mbox{ is a scalar matrix}\}.$$
As $A/\ker \pi \cong \MM_2(\C)$, it follows from Theorem \ref{thm:cqequiv} and the proofs of Lemma \ref{lem:postmaxqot} and Corollary \ref{cor:postlim} that $\cq(A)$ is not closed under addition and is not norm-closed, and there is $x \in \cq(A)$ such that $x^2 \notin \cq(A)$. To show this explicitly, first by Theorem \ref{thmcqel} we have
$$V_A=A \setminus \cq(A)=\{a \in A  : \exists \lm, \mu \in \C, \, \lm \neq \mu, \mbox{ such that } a(1)=\di(\lm, \lm, \mu)\}.$$
In particular, $\cq(A)$ strictly contains $Z(A)+J_{wc}(A)$. Let $b:=\di(1,0,0)$ and $c:=\di(0,1,0)$ be elements of $A$, considered as constant functions. Then, $b,c \in \cq(A)$, but $b+c=\di(1,1,0)\notin \cq(A)$. Similarly,  the constant function $x:=\di(-1,1,0)$ belongs to $\cq(A)$ but $x^2=\di(1,1,0)$ does not.

We now show that $\cq(A)$ is not norm-closed in $A$. In fact, we shall show that $\cq(A)$ is norm-dense in $A$, so as $A$ is not weakly central, $\cq(A)$ cannot be norm-closed (for a more general argument see Proposition \ref{propcqden}). Choose any $a \in A \setminus \cq(A)$. Then $a(1)=\di(\lm,\lm,\mu)$ for some distinct scalars $\lm$ and  $\mu$. For any $\ep >0$, let $b_\ep:=\di(\ep,0,0)$ (as a constant function in $A$). Then $a+b_\ep \in \cq(A)$ and $\|(a+b_\ep)-a\|=\|b_\ep\|=\ep$. 
\end{example}

We now demonstrate that Corollary \ref{cor:postlim} can fail when $A$ is not assumed to be postliminal or an AF-algebra. In order to do this, first recall that a $C^*$-algebra $B$ is said to be \emph{projectionless} if $B$ does not contain non-trivial projections. The first example of a simple projectionless $C^*$-algebra was given by Blackadar \cite{Bla1} (the non-unital example) and 
\cite{Bla2} (the unital example). Also, the prominent examples of simple projectionless $C^*$-algebras include the reduced $C^*$-algebra $C_r^*(\mathbb{F}_n)$ for the free group $\mathbb{F}_n$ on $n< \infty$ generators \cite{PV} and the Jiang-Su algebra $\mathcal{Z}$ \cite{JiSu}, which also has the important property that it is KK-equivalent to $\C$.  

\begin{lemma}\label{lem:seppoly}
Let $B$ be a unital  projectionless $C^*$-algebra and let $p \in \C[z]$ be a separable polynomial. An element $b \in B$ satisfies $p(b)=0$ if and only if $b=\mu 1$, where $\mu$ is a root of $p$.  
\end{lemma}
\begin{proof}
First note that since $B$ is projectionless, all elements of $B$ have connected spectrum. Indeed, otherwise by \cite[Corollary~3.3.7]{KR} $B$ would contain a non-trivial idempotent $e$ and then by \cite[Proposition~4.6.2]{BlaKT}, $e$ would be similar to a (necessarily non-trivial) projection. 

If $p\in \C[z]$ is a separable polynomial of degree $n$, we can factorize
$$p(z)=\alpha (z-\mu_1) \cdots (z-\mu_n),$$
where $\alpha \in \C \setminus \{0\}$ and $\mu_1, \ldots, \mu_n \in \C$ are distinct roots of $p$. If $b \in B$ satisfies $p(b)=0$, then the spectral mapping theorem implies $\s(b)\subseteq \{\mu_1, \ldots, \mu_n\}$. As $\s(b)$ is connected,  this forces $\s(b)=\{\mu_k\}$ for some $1 \leq k \leq n$.
Then for all $i \in \{1, \ldots, n\}\setminus \{k\}$, the element $b-\mu_i 1$ is invertible so
$$0=p(b)=\alpha(b-\mu_1 1)\cdots (b-\mu_n 1)$$
implies $b=\mu_k 1$ as claimed. The converse is trivial. 
\end{proof}

\begin{example}\label{ex:JiSutp}
Let $B$ be any unital simple projectionless non-abelian $C^*$-algebra (e.g. $B=\Z$, the Jiang-Su algebra). 

Consider the $C^*$-algebra $C$ of all continuous functions $x: [0,1]\to \MM_2(B) $ such that $x(1)=\di(b(x),0)$ for some $b(x)\in B$ (note that $C$ can be identified with the tensor product 
$A \otimes B$, where $A$ is the $C^*$-algebra from Example \ref{ex2subnqc}, which is nuclear). As $B$ is unital and simple, $Z(B)=\C1_B$, so 
$$Z(C)=\{\di(f 1_B, f 1_B) : \ f \in C([0,1]), \, f(1)=0\},$$
where $(f 1_B)(t)=f(t) 1_B$, for all $t \in [0,1]$. Consider the ideal $M$ of $C$ defined by 
$$
M:=\{x \in C : \, x(1)=0\}=C_0([0,1), \MM_2(B)).
$$
As $C/M \cong B$, $M$ is a modular maximal ideal of $C$ that contains $Z(C)$, so that $M \in T_C^1$. Since $Z(M) \cong C_0([0,1))$ and $\P(M)$ is canonically homeomorphic to $[0,1)$, it is easy to check directly that $M$ is a central $C^*$-algebra (alternatively, $M\cong C_0([0,1))\otimes \MM_2(B)$ is weakly central by Theorem \ref{thm:tenprod}). Therefore,
$$\qquad J_{wc}(C)=M \qquad \mbox{and} \qquad T_{C}=T_{C}^1=\{M\}.$$
By Theorem \ref{thmcqel} we have
$$\cq(C)=\{x \in C : \, b(x) \mbox{ is not a non-zero scalar}\}.$$
As $C/J_{wc}(C)\cong B$ is non-abelian, by Theorem \ref{thm:cqequiv} $\cq(C)$ is not norm-closed and is neither closed under addition nor closed under multiplication.

On the other hand we claim that for any $x \in \cq(C)$, $x^n \in \cq(C)$ for all $n \in \N$. On a contrary, assume that there exists $x \in \cq(C)$ such that $x^n \notin \cq(C)$ for some $n>1$. Then, by Theorem \ref{thmcqel}, there is a non-zero $\lambda \in \C$ such that $b(x)^n=\lambda 1_B$. Consider the polynomial $p(z):=z^n - \lambda$. As  $\lambda \neq 0$, $p$ is separable. Since $B$ is projectionless and $p(b(x))=0$,  Lemma \ref{lem:seppoly} implies that $b(x)=\mu 1_B$, where $\mu$ is some $n$-th root of $\lambda$. But this contradicts the fact that $x \in \cq(C)$.
\end{example}

If a unital $C^*$-algebra $A$ is not weakly central then, even though $\cq(A)$ might be a $C^*$-subalgebra of $A$ (and hence equal to $Z(A)+J_{wc}(A)$ by Theorem \ref{thm:cqequiv}), one may use matrix units to show that $\cq(\MM_2(A))$ is neither closed under addition nor closed under multiplication (for the algebraic counterpart, see the comment following \cite[Remark~3.6]{BG}). In fact, this is a special case of the following more general result.

\begin{proposition}\label{prop:tenprod}
Let $A$ be a unital $C^*$-algebra and let $B$ be a unital simple exact $C^*$-algebra. 
\begin{itemize}
\item[(a)] $J_{wc}(A\otimes_{\min} B)=J_{wc}(A)\otimes_{\min} B$.
\item[(b)] Suppose that $A$ is not weakly central and that $B$ is not abelian (that is, $B$ is not $*$-isomorphic to $\C$). Then $\cq(A\otimes_{\min} B)$ is not norm-closed and is neither closed under addition nor closed under  multiplication. In particular, $\cq(\MM_n(A))$ is not a $C^*$-subalgebra of $\MM_n(A)$ for any $n>1$.
\end{itemize}
\end{proposition}
\begin{proof}
(a) If $A$ is weakly central then, since $B$ is weakly central, we have that
$A\otimes_{\min} B$ is weakly central (see \cite[Theorem~3.1]{Arc2} and Theorem \ref{thm:tenprod}). So we now assume that $A$ is not weakly central, so that $T_A \neq \emptyset$. Let $M \in T_A$. Then there is $N \in \M(A)$ such that $N \neq M$ and $M\cap Z(A)=N \cap Z(A)$. Since $B$ is exact, 
$$\frac{A \otimes_{\min} B}{M \otimes_{\min} B} \cong \frac{A}{M} \otimes_{\min} B,$$
which is a simple $C^*$-algebra (see \cite[Corollary]{Tak}). Thus $M \otimes_{\min} B \in \M(A \otimes_{\min} B)$ and similarly  $N \otimes_{\min} B \in \M(A \otimes_{\min} B)$. Let $x \in (M \otimes_{\min} B)\cap (Z(A)\otimes \C1_B)$. For a state $\omega\in \S(B)$ let $L_\omega : A \otimes_{\min} B \to A$ be the corresponding left slice map (i.e. $L_\omega(a \otimes b)=\omega(b)a$, see \cite{Wass}). There exists $z \in Z(A)$ such that $x=z \otimes 1_{B}$ and hence $z=L_\omega(x)\in M$. Thus
\begin{eqnarray*}
(M \otimes_{\min} B) \cap (Z(A)\otimes \C1_B) &=& (M \cap Z(A)) \otimes \C 1_B =  (N \cap Z(A)) \otimes \C 1_B \\
&=&(N \otimes_{\min} B) \cap (Z(A)\otimes \C1_B).
\end{eqnarray*}
Note that also $M \otimes_{\min} B \neq N \otimes_{\min} B$ (for otherwise, by using $L_\omega$, we would obtain $M \subseteq N$ and $N \subseteq M$). Since by \cite[Corollary~1]{HW}, $Z(A)\otimes \C 1_B =Z(A\otimes_{\min} B)$, we have shown that $M \otimes_{\min} B \in T_{A\otimes_{\min} B}$.  By Theorem \ref{thmJcq}
\begin{equation}\label{eq:tenpr}
J_{wc}(A \otimes_{\min} B) \subseteq \bigcap_{M \in T_A} (M \otimes_{\min} B)=J_{wc}(A)\otimes_{\min} B.
\end{equation}
For the equality in (\ref{eq:tenpr}), let $y \in \bigcap_{M \in T_A} (M \otimes_{\min} B)$ and $\psi \in B^*$. Then 
$$L_\psi(y)\in \bigcap_{M \in T_A} M =J_{wc}(A).$$
Hence
$$0 = q(L_{\psi}(y)) = {\mathcal L}_{\psi}((q \otimes \mathrm{id}_B)(y)),$$
where $q: A \to A/J_{wc}(A)$ is the canonical map and ${\mathcal L}_{\psi}: (A/J_{wc}(A))  \otimes_{\min} B \to A/J_{wc}(A)$ is the left slice map. It follows that 
$$y \in \ker (q \otimes  \mathrm{id}_B) = J_{wc}(A) \otimes_{\min} B,$$ 
since $B$ is exact.

On the other hand, it follows from Theorem \ref{thm:tenprod} and Corollary \ref{corCQWC} that $J_{wc}(A) \otimes_{\min} B$ is weakly central. Thus $J_{wc}(A\otimes_{\min} B)=J_{wc}(A)\otimes_{\min} B$, as claimed.

\smallskip

(b) Since $B$ is exact, by (a)
$$\frac{A \otimes_{\min} B}{J_{wc}(A \otimes_{\min} B)}=\frac{A \otimes_{\min} B}{J_{wc}(A)\otimes_{\min} B}\cong \frac{A}{J_{wc}(A)} \otimes_{\min} B,$$
which is non-abelian. The result now follows from Theorem \ref{thm:cqequiv}.
\end{proof}

\smallskip

In contrast to the second paragraph of Remark \ref{rem:fin} we now demonstrate there are even separable continuous-trace $C^*$-algebras $A$ such that $Z(A)=J_{wc}(A)=\{0\}$, while $\cq(A)$ is norm-dense in $A$.  In order to do this, we shall use the following facts.

\begin{lemma}\label{lemma:primdense}
Let $A$ be a $C^*$-algebra such that all primitive ideals of $A$ are maximal and both sets of all modular and non-modular primitive ideals are dense in $\P(A)$. Then $Z(A)=J_{wc}(A)=\{0\}$.
\end{lemma}
\begin{proof}
That $Z(A)=\{0\}$ follows from Remark \ref{nmod}. Let $I$ be a non-zero ideal of $A$. Then $Z(I)= I\cap Z(A) =\{0\}$. On the other hand, the dense set of modular primitive ideals of $A$ meets the open set $\P_I(A)$. If $P$ is any modular primitive ideal of $A$ that does not contain $I$ then, by assumption, $P$ is maximal,  so by Lemma \ref{lemmaxJ} $P \cap I$ is a modular primitive ideal of $I$ such that $\{0\}=Z(I) \subseteq P \cap I$. Therefore, $I$ is not weakly central.
\end{proof}

\begin{proposition}\label{propcqden} 
Let $A$ be a $C^*$-algebra.
\begin{itemize}
\item[(a)] If either there is $M \in \M(A)$ of codimension $1$ such that $Z(A)\subseteq M$ or there are distinct $M_1,M_2 \in \M(A)$ of codimension $1$ that satisfy $M_1 \cap Z(A)=M_2 \cap Z(A)\neq Z(A)$, then  $\cq(A)$ is not norm-dense in $A$.
\item[(b)] The converse of (a) is true if $T_A$ is countable.
\end{itemize}
\end{proposition}
\begin{proof}
(a) Assume there is $M \in \M(A)$ of codimension $1$ that contains $Z(A)$. Since $A/M\cong \C$, by Theorem \ref{thmcqel} for any $a \in \cq(A)$, $a+M$ is zero in $A/M$, so $a \in M$. Thus, $\cq(A)\subseteq M$, so $\cq(A)$ is clearly not norm-dense in $A$.

Alternatively, assume there are distinct $M_1,M_2 \in \M(A)$ of codimension $1$ such that $M_1 \cap Z(A)=M_2 \cap Z(A)\neq Z(A)$. Since $A/ (M_1 \cap M_2) \cong (A/M_1) \oplus (A/M_2) \cong \C \oplus \C$, Theorem \ref{thmcqel} implies $\cq(A)\subseteq \C 1+(M_1 \cap M_2)$, so $\cq(A)$ is not norm-dense in $A$.

(b) Now assume that all $M \in \M(A)$ that contain $Z(A)$ have codimension greater than $1$ and for all distinct $M_1,M_2 \in \M(A)$ that satisfy $M_1 \cap Z(A)=M_2 \cap Z(A)\neq Z(A)$, at least one $M_i$ has codimension greater than $1$. We may assume that $T_A\neq \emptyset$, for otherwise $\cq(A) = A$, which is certainly dense in $A$.

For each $M \in T_A$ such that $\dim(A/M)>1$, set 
$$U_M:=\{a \in A : \, a+M  \mbox{ is not a scalar in } A/M\}.$$
Evidently, $U_M$ is an open subset of $A$. We claim that $U_M$ is norm-dense in $A$. Let $a\in A \setminus U_M$, so that $a+M$ is a scalar in $A/M$. Let $\ep >0$. Since $A/M$ is non-abelian, there is a non-central element $\dot{b}$ of norm one in $A/M$. Then by \cite[Lemma~17.3.3]{WO}, there is a norm one element $b \in A$ such that $b+M=\dot{b}$. Then the element $a+ (\ep/2)b$ lies in  $U_M$ and its distance from $a$ is $\ep/2$.  

If $T_A$ is countable, then the Baire category theorem implies that 
$$U:=\bigcap\{U_M: \, M \in T_A, \, \dim(A/M)>1\}$$
is a dense subset of $A$. Let $a \in U$. If $M \in T^1_A$ then $a \in U_M$ and so $a+M$ is not a scalar in $A/M$. Also if  $M_1,M_2 \in \M(A)$, such that $M_1 \neq M_2$ and $M_1 \cap Z(A)=M_2 \cap Z(A)\neq Z(A)$, then for some $i \in \{1,2\}$ we have $\dim(A/M_i)>1$ so that $a\in U_{M_i}$ and hence $a+M_i$ is not a scalar in $A/M_i$. Thus, by Theorem \ref{thmcqel},  $U \subseteq \cq(A)$, so $\cq(A)$ is norm-dense in $A$.
\end{proof}

The next example is a slight variant of \cite[Example~4.4]{ArcTh} where we have changed the quotient $A(1)$ in order to avoid an abelian quotient.

\begin{example}\label{exRobTh}
Let $\H$ be a separable infinite-dimensional Hilbert space with orthonormal basis $\{e_n : \, n \geq 0\}$. For each $n$ let $E_n$ be the projection from $\H$ onto the linear span of the set $\{e_0, e_1, \ldots, e_n\}$. We define
$A$ to be the subset of $C([0,1],\KK(\H))$ consisting of all elements $a \in C([0,1],\KK(\H))$ which satisfy the following requirement: For any dyadic rational $t=p/2^q\in [0,1)$, where $p,q$ are positive integers such that $2 \not | p$, then
$$a(t)=E_qa(t)=a(t)E_q.$$
Then $A$ is a closed self-adjoint subalgebra of $C([0,1],\KK(\H))$ and so is itself a $C^*$-algebra.

As in  \cite{ArcTh}, standard arguments show that $A$ is a continuous-trace $C^*$-algebra whose primitive ideal space can be identified with $[0,1]$, via the homeomorphism
$$[0,1]\ni t \mapsto P_t:=\ker \pi_t\in \P(A),$$ where for each $t \in [0,1]$ and $a \in A$, $\pi_t(a):=a(t)$. Moreover, if for each $t \in [0,1]$ we denote the fibre of $A$ at $t$ by $A(t)$ (i.e. $A(t)=\{a(t) : \, a \in A\}$), then
$$
A(t) =
\begin{cases}
 \{K \in \KK(\H) : \, E_qK=KE_q=K\}\cong \MM_{q+1}(\C), & \mbox{if } t=p/2^q \mbox{ as above} \\
 \KK(\H), & \mbox{otherwise}.
\end{cases}
$$
In particular, all primitive ideals of $A$ are maximal. Further, the sets of modular and non-modular primitive ideals of $A$ are both dense in $\P(A)$, and so Lemma \ref{lemma:primdense} implies $Z(A)=J_{wc}(A)=\{0\}$.
On the other hand, since
$$T_A=T_A^1=\{P_t : \, t  \in [0,1) \mbox{ is a dyadic rational}\}$$
is countable and the codimension of each $P_t \in T_A$ is larger than $1$, Proposition \ref{propcqden} implies that $\cq(A)$ is norm-dense in $A$. 
\end{example}

\section*{Funding}
The second-named author was fully supported by the Croatian Science Foundation under the project IP-2016-06-1046.

\end{document}